\documentclass{birkjour}

\usepackage{amsmath}
\usepackage{amssymb}
\usepackage{verbatim}   
\usepackage{color}    
\usepackage{mathrsfs,epsfig}  
\usepackage{fontenc}
\usepackage{textcomp}
\usepackage{fix-cm}
\usepackage{array}
\usepackage{enumerate}
\usepackage{latexsym, amsfonts, amssymb}
\usepackage{color}
\usepackage{comment}
\usepackage{subfigure}
\usepackage{accents}

\numberwithin{equation}{section}

\newtheorem{theorem}[equation]{Theorem}

\newtheorem{lemma}[equation]{Lemma}
\newtheorem{corollary}[equation]{Corollary}

\theoremstyle{definition}
\newtheorem{definition}[equation]{Definition}
\newtheorem{remark}[equation]{Remark}

\begin{document}
\fontsize{9}{11}\selectfont

\title[Generalized Inviscid Proudman-Johnson Equation]{The Role of Initial Curvature in Solutions to the Generalized Inviscid Proudman-Johnson Equation}

\author{Alejandro Sarria}
\address{%
Department of Mathematics\\
University of New Orleans\\
New Orleans, LA, 70148, USA}
\email{asarria1@uno.edu}

\author{Ralph Saxton}
\address{%
Department of Mathematics\\
University of New Orleans\\
New Orleans, LA, 70148, USA
}
\email{rsaxton@uno.edu}

\subjclass{35B44, 35B10, 35B65, 35Q35}

\keywords{Proudman-Johnson equation, blow-up.}

\begin{abstract}
In \cite{Sarria1}, we derived representation formulae for spatially periodic solutions to the generalized, inviscid Proudman-Johnson equation and studied their regularity for several classes of initial data. The purpose of this paper is to extend these results to larger classes of functions including those having arbitrary local curvature near particular points in the domain. 
\end{abstract}

\maketitle

\section{Introduction}\hfill
\label{sec:intro}

In this article, we extend the analysis initiated in \cite{Sarria1} concerning blow-up, and blow-up properties, in solutions to the initial boundary value problem for the generalized, inviscid Proudman-Johnson equation (\cite{Proudman1}, \cite{Childress}, \cite{Okamoto1})
\begin{equation}
\label{eq:nonhomo}
\begin{cases}
u_{xt}+uu_{xx}-\lambda u_x^2=I(t),\,\,\,\,t>0,
\\
u(x,0)=u_0(x),\,\,\,\,\,\,\,\,\,\,\,\,\,\,\,\,\,\,\,\,\,\,\,\,\,\,\,x\in[0,1],
\\
I(t)=-(\lambda+1)\int_0^1{u_x^2\,dx},
\end{cases}
\end{equation}
where $\lambda\in\mathbb{R}$ and solutions are subject to periodic boundary conditions
\begin{equation}
\label{eq:pbc}
\begin{split}
u(0,t)=u(1,t),\,\,\,\,\,\,u_x(0,t)=u_x(1,t).
\end{split}
\end{equation} 

We note that the equation arises in several important applications, in the presence or absence of the nonlocal term $I(t)$. For $\lambda=-1,$ (\ref{eq:nonhomo})i), iii) reduces to the inviscid Burgers' equation of gas dynamics differentiated once in space. If $\lambda=-1/2,$ the Hunter Saxton equation (HS) describes the orientation of waves in a massive director field of a nematic liquid crystal (\cite{Hunter1}, \cite{Bressan1}, \cite{Dafermos1}, \cite{Yin1}). For periodic functions, the HS equation also has a deep geometric meaning as it describes geodesics on a group of orientation preserving diffeomorphisms on the unit circle modulo rigid rotations with respect to a right-invariant metric (\cite{Khesin1}, \cite{Bressan1}, \cite{Tiglay1}, \cite{Lenells1}). If $\lambda=\frac{1}{n-1},\,n\geq2,$ (\ref{eq:nonhomo}) i), iii) can be obtained directly from the $n-$dimensional incompressible Euler equations 

$$\boldsymbol{u}_t+(\boldsymbol{u}\cdot\nabla)\boldsymbol{u}=-\nabla p,\,\,\,\,\,\,\,\,\nabla\cdot \boldsymbol{u}=0$$ 

using stagnation point-form velocities $\boldsymbol{u}(x,\boldsymbol{x}^\prime,t)=(u(x,t),-\lambda\boldsymbol{x}^\prime u_x(x,t))$,\, $\boldsymbol{x}^\prime=\{x_2,...,x_n\},$ or through the cylindrical coordinate representation  $u^r=-\lambda ru_x(x,t),$ $u^{\theta}=0$ and $u^x=u(x,t)$, where $r=\left|\boldsymbol{x}^\prime\right|,$ (\cite{Childress}, \cite{Weyl1}, \cite{Saxton1}, \cite{Okamoto1}, \cite{Escher1}). Finally, in the local case $I(t)=0$, the equation appears as a special case of Calogero's equation 
$$u_{xt}+uu_{xx}-\Phi(u_x)=0$$ 
for arbitrary functions $\Phi(\cdot)$ (\cite{Calogero1}). 

In \cite{Sarria1} we derived representation formulae for periodic solutions to (\ref{eq:nonhomo})-(\ref{eq:pbc}) and, for several classes of mean-zero initial data, examined their $L^p$ regularity for $p\in[1,+\infty]$. For convenience of the reader, the main results established in \cite{Sarria1} are summarized in Theorems \ref{thm:sarria1}-\ref{thm:sarria2} below.

\begin{theorem}
\label{thm:sarria1} 
Consider the initial boundary value problem (\ref{eq:nonhomo})-(\ref{eq:pbc}). There exist smooth, mean zero initial data such that: 
\begin{enumerate}
\item\label{it:1} For $\lambda\in(-\infty,-2]\cup(1,+\infty)$, there is a finite $t_*>0$ such that $\lim_{t\uparrow t_*}\left|u_x(x,t)\right|=+\infty$ for every $x\in[0,1]$. Additionally, the blow-up is two-sided (two-sided, everywhere blow-up). 
\item\label{it:2} For $\lambda\in(-2,0),$ there is a finite time $t_*>0$ and a finite number of $\underline x_j\in[0,1]$, $j\in\mathbb{N}$, such that $\lim_{t\uparrow t_*}u_x(\underline x_j,t)=-\infty$ (one-sided, discrete blow-up). 
\item\label{it:3} For $\lambda\in[0,1],$ solutions persist globally in time. More particularly, these vanish as $t\uparrow t_*=+\infty$ for $\lambda\in(0,1)$ but converge to a nontrivial steady state for $\lambda=1.$
\end{enumerate}
\end{theorem}
For $t_*>0$ as in Theorem \ref{thm:sarria1} above, Theorem \ref{thm:lpintro} below examines $L^p(0,1)$ regularity of $u_x$ for $t\in[0,t_*)$ and $p\in[1,+\infty)$.

\begin{theorem}
\label{thm:lpintro}
Let $u$ in Theorem \ref{thm:sarria1} be a solution to the initial boundary value problem (\ref{eq:nonhomo})-(\ref{eq:pbc}) defined for $t\in[0,t_*)$. Then
\begin{enumerate}
\item\label{it:lp2} For $p\geq1$ and $\frac{2}{1-2p}<\lambda\leq1,\, \lim_{t\uparrow t_*}\left\|u_x\right\|_p<+\infty.$
\item\label{it:lp1} For $p\in(1,+\infty)$ and $\lambda\in(-\infty,-2/p]\cup(1,+\infty)$,\, $\lim_{t\uparrow t_*}\left\|u_x\right\|_p=+\infty$. 
\item\label{it:ener} The energy $E(t)=\left\|u_x\right\|_2^2$ diverges if $\lambda\in\mathbb{R}\backslash(-2/3,1]$ as $t\uparrow t_*$ but remains finite for $t\in[0,t_*]$ otherwise. Moreover, $\dot E(t)$ blows up to $+\infty$ as $t\uparrow t_*$ when $\lambda\in\mathbb{R}\backslash[-1/2,1]$ and $\dot E(t)\equiv0$ for $\lambda=-1/2$; whereas, $\lim_{t\uparrow t_*}\dot E(t)=-\infty$ if $\lambda\in(-1/2,-2/5]$ but remains bounded when $\lambda\in(-2/5,1]$ for all $t\in[0,t_*]$. 
\end{enumerate} 
\end{theorem}
See \S\ref{sec:preliminaries} for details on the class of initial data used to establish Theorems \ref{thm:sarria1} and \ref{thm:lpintro}. Lastly, let $PC_{\mathbb{R}}(0,1)$ denote the family of piecewise constant functions with zero mean in $[0,1]$. Then, in \cite{Sarria1} we proved the following:

\begin{theorem}
\label{thm:sarria2}
For the initial boundary value problem (\ref{eq:nonhomo})-(\ref{eq:pbc}),
\begin{enumerate}

\item\label{it:4} Suppose $u_0''(x)\in PC_{\mathbb{R}}(0,1)$ and $\lambda>1/2$. Then, there exist solutions and a finite $t_*>0$ for which $u_x$ undergoes a two-sided, everywhere blow-up as $t\uparrow t_*$. If $\lambda<0,$ a one-sided discrete blow-up may occur instead. In contrast, for $\lambda\in[0,1/2]$, solutions may persist globally in time. More particularly, these either vanish as $t\uparrow t_*=+\infty$ if $\lambda\in(0,1/2)$, or converge to a nontrivial steady-state for $\lambda=1/2$.
	
\item\label{it:5} Suppose $u_0'(x)\in PC_{\mathbb{R}}(0,1)$ and assume solutions are defined for all $t\in[0,T],\, T>0.$ Then no  $W_{\mathbb{R}}^{1,\infty}(0,1)$ solution may exist for $T\geq t_*$, where $0<t_*<+\infty$ if $\lambda<0$, and $t_*=+\infty$ for $\lambda\geq0$. Further, $\lim_{t\uparrow t_*}\left\|u_x\right\|_1=+\infty$ when $\lambda<-1$, while 
\begin{equation*}
\lim_{t\uparrow t_*}\left\|u_x\right\|_p=
\begin{cases}
C,\,\,\,\,\,\,\,\,&-\frac{1}{p}\leq\lambda<0,\,\,\,\,\,\,\,\,\,p\geq1,
\\
+\infty,\,\,\,\,\,\,\,\,&-1\leq\lambda<-\frac{1}{p},\,\,\,\,p>1,
\end{cases}
\end{equation*}
where the constants $C\in\mathbb{R}^+$ depend on the choice of $\lambda$ and $p.$
\end{enumerate}
\end{theorem}
The reader may refer to \cite{Sarria1} for details, and the works \cite{Okamoto2}, \cite{Aconstantin1}, \cite{Hunter2}, \cite{Wunsch2}, \cite{Wunsch3}, \cite{Wunsch1} for additional background. The purpose of this work is to extend the above results to initial data which belongs to classes of functions with varying concavity profile near certain points in the domain. More particularly, we suppose throughout that $u_0'(x)$ is bounded and, at least, $C^0(0,1)$ $a.e.$ Then, for $\lambda>0$, we will assume there are constants $q, M_0\in\mathbb{R}^+$ and $C_1\in\mathbb{R}^-$, and a finite number of points $\overline\alpha_i\in[0,1]$ such that, near  $\overline\alpha_i$,
\begin{equation}
\label{eq:expnew02}
u_0^\prime(\alpha)\sim M_0+C_1\left|\alpha-\overline\alpha_i\right|^q.
\end{equation}
Analogously, for $\lambda<0$, we suppose there are constants $C_2\in\mathbb{R}^+$, $m_0\in\mathbb{R}^-$, and a finite number of locations $\underline\alpha_j\neq\overline\alpha_i$ in $[0,1]$ such that, in a neighbourhood of $\underline\alpha_j$,
\begin{equation}
\label{eq:expnew002}
u_0^\prime(\alpha)\sim m_0+C_2\left|\alpha-\underline\alpha_j\right|^q.
\end{equation}
We refer to \S\ref{subsec:dataclass} for specifics of the above. It is worth mentioning that, for $q\in(0,1)$, the above local estimates may lead to cusps in the graph of $u_0'$, therefore possible jump discontinuities in $u_0''$ of \textsl{infinite} magnitude across $\overline\alpha_i$ and/or $\underline\alpha_j$. In contrast, a jump discontinuity of \textsl{finite} magnitude in $u_0''$ may occur if $q=1$. As we will see in the coming sections, the  finite or infinite character in the size of this jump plays a decisive role, particularly in the formation of spontaneous singularities for the special case of stagnation point-form solutions to the three dimensional incompressible Euler equations.

The remaining of the paper is organized as follows. In \S\ref{sec:preliminaries}, we provide an outline for the derivation of the representation formulae established in \cite{Sarria1} and provide further details on the class of initial data to be considered in this article. Then, new blow-up results are stated and proved in \S\ref{sec:blowup}, while specific examples are to be found in \S\ref{sec:examples}.

\section{Preliminaries}
\label{sec:preliminaries}

\subsection{The General Solution}\hfill
\label{subsec:sol}

In \cite{Sarria1}, we used the method of characteristics to derive a representation formula for periodic solutions to (\ref{eq:nonhomo}). For convenience of the reader, below we give a brief outline of the derivation. 

Define the characteristics, $\gamma,$ as the solution to the initial value problem
\begin{equation}
\label{eq:cha}
\dot\gamma(\alpha,t)=u(\gamma(\alpha,t),t),\,\,\,\,\,\,\,\,\,\,\,\,\gamma(\alpha,0)=\alpha\in[0,1],
\end{equation}
so that
\begin{equation}
\label{eq:jacid}
\begin{split}
\dot\gamma_{\alpha}(\alpha,t)=u_x(\gamma(\alpha,t),t)\cdot\gamma_{\alpha}(\alpha,t).
\end{split}
\end{equation}
Then, using (\ref{eq:nonhomo})i), iii) and the above, we obtain 
\begin{equation}
\label{eq:chain}
\begin{split}
\ddot\gamma_\alpha&=(u_{xt}+uu_{xx})\circ\gamma\cdot\gamma_\alpha+(u_x\circ\gamma)\cdot\dot\gamma_\alpha
\\
&=(u_{xt}+uu_{xx})\circ\gamma\cdot\gamma_\alpha+u_x^2\circ\gamma\cdot\gamma_\alpha
\\
&=(\lambda+1)\left(u_x^2\circ\gamma-\int_0^1{u_x^2dx}\right)\cdot\gamma_\alpha
\\
&=(\lambda+1)\left((\gamma^{-1}_{\alpha}\cdot\dot\gamma_{\alpha})^2-\int_0^1{u_x^2dx}\right)\cdot\gamma_\alpha\,,
\end{split}
\end{equation}
which for $\lambda\neq0$, $I(t)=-(\lambda+1)\int_0^1{u_x^2dx}$, and $\omega(\alpha,t)=\gamma_{\alpha}(\alpha,t)^{-\lambda}$, can be written as
\begin{equation}
\label{eq:nonhomo2}
\begin{split}
\ddot\omega(\alpha,t)+\lambda I(t)\omega(\alpha,t)=0.
\end{split}
\end{equation}
Assume we have two linearly independent solutions $\phi_1(t)$ and $\phi_2(t)$ to (\ref{eq:nonhomo2}) satisfying $\phi_1(0)=\dot{\phi}_2(0)=1$ and $\dot{\phi}_1(0)=\phi_2(0)=0$. Then, since $\dot{\omega}=-\lambda\gamma_\alpha^{-(\lambda+1)}\dot{\gamma_\alpha}$ and $\gamma_\alpha(\alpha,0)=1$, we deduce that 
\begin{equation}
\label{eq:compat}
\begin{split}
\omega(\alpha,t)=\phi_1(t)\left(1-\lambda \eta(t)u_0'(\alpha)\right),\,\,\,\,\,\,\,\,\,\,\,\,\eta(t)=\int_0^t\frac{ds}{\phi_1^2(s)}.
\end{split}
\end{equation}
Now, uniqueness of solution to (\ref{eq:cha}) and periodicity implies that
\begin{equation}
\label{eq:chaperio}
\begin{split}
\gamma(\alpha+1,t)-\gamma(\alpha,t)=1
\end{split}
\end{equation} 
for as long as $u$ is defined. Consequently, simplifying and integrating (\ref{eq:compat})i) with respect to $\alpha$ gives
\begin{equation}
\label{eq:sum}
\gamma_\alpha={\mathcal K}_0/{\bar{\mathcal K}}_0
\end{equation}
where we define
\begin{equation}
\label{eq:def}
\begin{split}
\mathcal{K}_i(\alpha, t)=\frac{1}{\mathcal{J}(\alpha,t)^{i+{\frac{1}{\lambda}}}},\,\,\,\,\,\,\,\,\,\,\,\,\bar{\mathcal{K}}_i(t)=\int_0^1{\frac{d\alpha}{\mathcal{J}(\alpha,t)^{i+\frac{1}{\lambda}}}},
\end{split}
\end{equation}
for $i\in\mathbb{N}\cup\{0\}$, and
\begin{equation}
\label{eq:J}
\begin{split}
\mathcal{J}(\alpha,t)=1-\lambda\eta(t)u_0^\prime(\alpha),\,\,\,\,\,\,\,\,\,\,\,\,\mathcal{J}(\alpha,0)=1.
\end{split}
\end{equation}
As a result, (\ref{eq:jacid}) and (\ref{eq:J})i) yield, after further simplification,
\begin{equation}
\label{eq:mainsolu}
\begin{split}
u_x(\gamma(\alpha,t),t)=\frac{1}{\lambda\eta(t){\bar{\mathcal K}_0(t)}^{2\lambda}}
\left(\frac{1}{\mathcal{J}(\alpha, t)}-\frac{\bar{\mathcal{K}}_1(t)}{\bar{\mathcal K}_0(t)}
\right).
\end{split}
\end{equation} 
The strictly increasing function $\eta(t)$ satisfies the initial value problem
\begin{equation}
\label{eq:etaivp}
\begin{split}
\dot{\eta}(t)=\bar{\mathcal{K}}_0(t)^{^{-2\lambda}},\,\,\,\,\,\,\,\,\,\,\,\eta(0)=0,
\end{split}
\end{equation}   
from which the existence of an eventual finite blow-up time $t_*>0$ for (\ref{eq:mainsolu}) will depend, in turn, upon the existence of a finite, positive limit
\begin{equation}
\label{eq:assympt}
\begin{split}
t_*\equiv\lim_{\eta\uparrow\eta_*}\int_0^{\eta}{\left(\int_0^1{\frac{d\alpha}{(1-\lambda\mu u_0^\prime(\alpha))^{\frac{1}{\lambda}}}}\right)^{2\lambda}\,d\mu}
\end{split}
\end{equation}
for $\eta_*>0$ to be defined. Moreover, assuming sufficient smoothness, (\ref{eq:sum}) and (\ref{eq:mainsolu}) imply that
\begin{equation}
\label{eq:preserv1}
\begin{split}
u_{xx}(\gamma(\alpha,t),t)=\frac{u_0^{\prime\prime}(\alpha)}{\mathcal{J}(\alpha,t)^{2-\frac{1}{\lambda}}}\bar{\mathcal{K}}_0(t)^{1-2\lambda},
\end{split}
\end{equation}
so that, for as long as it exists, $u$ maintains its initial concavity profile.

\subsection{The Data Classes}\hfill
\label{subsec:dataclass}

Suppose solutions exist for $t\in[0,t_*)$,\, $0<t_*\leq+\infty$. Define
\begin{equation}
\label{eq:max}
\begin{split}
M(t)\equiv\sup_{\alpha\in[0,1]}\{u_x(\gamma(\alpha,t),t)\},\,\,\,\,\,\,\,\,\,\,\,\,\,\,\,\,\,M(0)=M_0
\end{split}
\end{equation}
and
\begin{equation}
\label{eq:min22}
\begin{split}
m(t)\equiv\inf_{\alpha\in[0,1]}\{u_x(\gamma(\alpha,t),t)\},\,\,\,\,\,\,\,\,\,\,\,\,\,\,\,\,\,m(0)=m_0,
\end{split}
\end{equation}
where $\overline\alpha_i$, $i=1,2,...,m$, and $\underline\alpha_j$, $j=1,2,...,n$, denote the finite\footnote[1]{One possibility for having an infinite number of these points will be considered later on via a limiting argument.} number of locations in $[0,1]$ where $u_0'(\alpha)$ attains its greatest and least values $M_0>0>m_0$, respectively. Then, it follows from (\ref{eq:mainsolu}) (\cite{Sarria1}) that 
\begin{equation}
\label{eq:maxmin}
\begin{split}
M(t)=u_x(\gamma(\overline\alpha_i,t),t),\,\,\,\,\,\,\,\,\,\,\,\,\,\,\,\,m(t)=u_x(\gamma(\underline\alpha_j,t),t)
\end{split}
\end{equation}
for $0\leq t<t_*$. Now, the results of Theorems \ref{thm:sarria1}-\ref{thm:sarria2} suggest that the curvature of $u_0'$ near $\overline\alpha_i$ and/or $\underline\alpha_j$ plays a decisive role in the regularity of solutions to (\ref{eq:nonhomo}). Therefore, in the following sections, we further examine this interaction by considering a large class of functions in which $u_0'(x)$ is assumed to be bounded, at least $C^0(0,1)\, a.e.$, and has arbitrary curvature near the location(s) in question. More particularly, for $\lambda>0$, we will assume there are constants $q\in\mathbb{R}^+$ and $C_1\in\mathbb{R}^-$ such that 
\begin{equation}
\label{eq:expnew0}
u_0^\prime(\alpha)\sim M_0+C_1\left|\alpha-\overline\alpha_i\right|^q
\end{equation}
for $0\leq\left|\alpha-\overline\alpha_i\right|\leq r$, and small enough $0<r\leq1$, $r\equiv\min_{1\leq i\leq m}\{r_i\}$. Similarly, for $\lambda<0$, we suppose there is $C_2\in\mathbb{R}^+$ such that
\begin{equation}
\label{eq:expnew00}
u_0^\prime(\alpha)\sim m_0+C_2\left|\alpha-\underline\alpha_j\right|^q
\end{equation}
for $0\leq\left|\alpha-\underline\alpha_j\right|\leq s$ and $0<s\leq1$, $s\equiv\min_{1\leq j\leq n}\{s_j\}$.
See Figure \ref{fig:data0} below. Now, for $r$ and $s$ as above, define 
$$\mathcal{D}_i\equiv[\overline\alpha_i-r,\overline\alpha_i+r],\,\,\,\,\,\,\,\,\,\,\,\,\,\mathcal{D}_j\equiv[\underline\alpha_j-s,\underline\alpha_j+s].$$
Then, below we list some of the data classes that admit the asymptotic behaviour (\ref{eq:expnew0}) and/or (\ref{eq:expnew00}) for particular values of $q>0$. 

\begin{itemize}

\item $u_0(x)\in C^{\infty}(0,1)$ for $q=2k$ and $k\in\mathbb{Z}^+$ (see definition \ref{def:order}).

\item If $q=1,$ $u_0''(x)\in PC(\mathcal{D}_i)$ for $\lambda>0$, or $u_0''(x)\in PC(\mathcal{D}_j)$ if $\lambda<0$.

\item In the limit as $q\to+\infty$, $u_0'(x)\in PC(\mathcal{D}_i)$ for $\lambda>0$, or $u_0'(x)\in PC(\mathcal{D}_j)$ if $\lambda<0$.

\item From (\ref{eq:expnew0}), we see that the quantity

\begin{equation}
\label{eq:holder2}
[u_0']_{_{q;\overline\alpha_i}}=\sup_{\alpha\in\mathcal{D}_i}\frac{\lvert u_0'(\alpha)-u_0'(\overline\alpha_i)\rvert}{\lvert\alpha-\overline\alpha_i\rvert^q}
\end{equation}

is finite. As a result, for $0<q\leq1$ and $\lambda>0$, $u_0'$ is H$\ddot{\text{o}}$lder continuous at $\overline\alpha_i$. Analogously for $\lambda<0$, since

\begin{equation}
\label{eq:holder3}
[u_0']_{_{q;\underline\alpha_j}}=\sup_{\alpha\in\mathcal{D}_j}\frac{\lvert u_0'(\alpha)-u_0'(\underline\alpha_j)\rvert}{\lvert\alpha-\underline\alpha_j\rvert^q}
\end{equation}

is defined by (\ref{eq:expnew00}).

\item For $\lambda>0$ and either $N<q<N+1$, $N\in\mathbb{N}$, or $q>0$ odd, $u_0'(\alpha)\in C^{^{N+1}}(\mathcal{D}_i)$. Similarly for $\lambda<0$.
\end{itemize}

\begin{center}
\begin{figure}[!ht]
\includegraphics[scale=0.25]{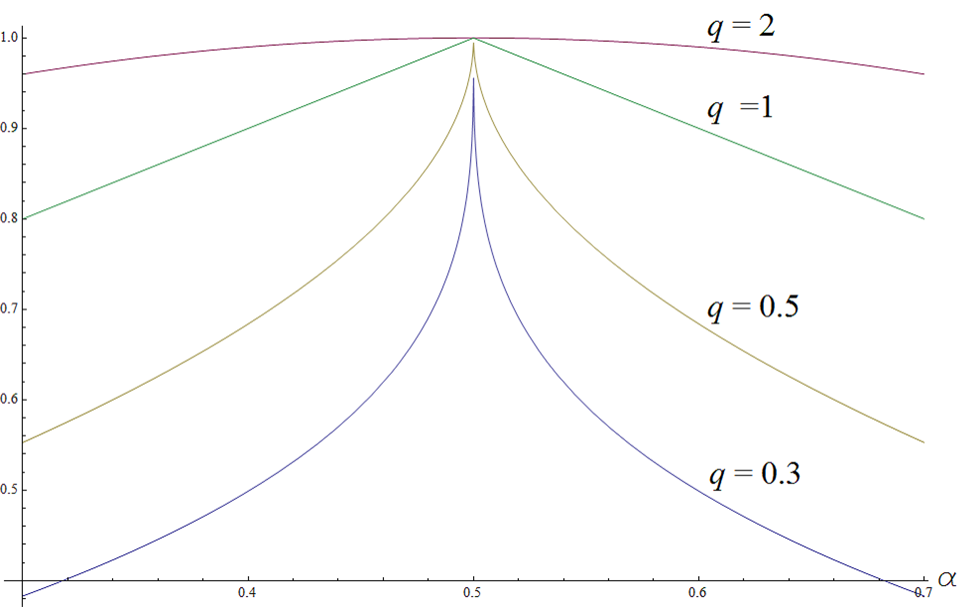} 
\caption{Local behaviour of $u_0'(\alpha)$ satisfying (\ref{eq:expnew0}) for several values of $q>0$, $\overline\alpha=1/2$, $M_0=1$ and $C_1=-1$.}
\label{fig:data0}
\end{figure}
\end{center}

\section{Blow-up}\hfill
\label{sec:blowup}

In this section, we study regularity properties in solutions to (\ref{eq:nonhomo})-(\ref{eq:pbc}) which, according to the sign of $\lambda$, arise from initial data satisfying (\ref{eq:expnew0}) and/or (\ref{eq:expnew00}). More particularly, finite-time blow-up and global existence in time are examined using $L^p(0,1)$ Banach spaces for $p\in[1,+\infty]$. Set
\begin{equation}
\label{eq:defeta*}
\begin{split}
\eta_*=
\begin{cases}
\frac{1}{\lambda M_0},\,\,\,\,\,\,\,&\lambda>0,
\\
\frac{1}{\lambda m_0},\,\,\,\,\,\,\,&\lambda<0.
\end{cases}
\end{split}
\end{equation}
Then, as $\eta\uparrow \eta_*,$ the space-dependent term in (\ref{eq:mainsolu}) will diverge for certain choices of $\alpha$ and not at all for others. Specifically, for $\lambda>0,$ $\mathcal{J}(\alpha,t)^{-1}$ blows up  earliest as $\eta\uparrow \eta_*$ at $\alpha=\overline\alpha_i,$ since
\begin{equation*}
\begin{split}
\mathcal{J}(\overline\alpha_i,t)^{-1}=(1-\lambda\eta(t)M_0)^{-1}\to+\infty\,\,\,\,\,\,\,\text{as}\,\,\,\,\,\,\,\eta\uparrow\eta_*=\frac{1}{\lambda M_0}.
\end{split}
\end{equation*}
Similarly for $\lambda<0,\, \mathcal{J}(\alpha,t)^{-1}$ diverges first at $\alpha=\underline\alpha_j$ and 
\begin{equation*}
\begin{split}
\mathcal{J}(\underline\alpha_j,t)^{-1}=(1-\lambda\eta(t)m_0)^{-1}\to+\infty\,\,\,\,\,\,\,\text{as}\,\,\,\,\,\,\,\eta\uparrow\eta_*=\frac{1}{\lambda m_0}.
\end{split}
\end{equation*}
However, blow-up of (\ref{eq:mainsolu}) does not necessarily follow from this; we will need to estimate the behaviour of the time-dependent integrals 
\begin{equation*}
\begin{split}
\bar{\mathcal{K}}_0(t)=
\int_0^1{\frac{d\alpha}{\mathcal{J}(\alpha,t)^{\frac{1}{\lambda}}}},\,\,\,\,\,\,\,\,\,\,\,\,\,\,\,\,\,\,\,\,\,\,\,\,\,\,\bar{\mathcal{K}}_1(t)=\int_0^1{\frac{d\alpha}{\mathcal{J}(\alpha,t)^{1+\frac{1}{\lambda}}}}
\end{split}
\end{equation*}
as $\eta\uparrow\eta_*.$ To this end, in some of the proofs we find convenient the use of the Gauss hypergeometric series (\cite{Barnes1}, \cite{Magnus1}, \cite{Gasper1})
\begin{equation}
\label{eq:2f1}
\begin{split}
{}_2F_1\left[a,b;c;z\right]\equiv\sum_{k=0}^{\infty}\frac{\left(a\right)_k(b)_k}{\left(c\right)_k\,k!}z^k,\,\,\,\,\,\,\,\,\,\,\,\lvert z\rvert< 1,
\end{split}
\end{equation} 
for $c\notin\mathbb{Z}^-\cup\{0\}$ and $(x)_k,\, k\in\mathbb{N}\cup\{0\}$, the Pochhammer symbol $(x)_0=1$, $(x)_k=x(x+1)...(x+k-1).$ Also, we will make use of the following results: 
\begin{lemma}
\label{lem:analcont}
Suppose $\lvert\text{arg}\left(-z\right)\rvert<\pi$ and $a,b,c,a-b\notin\mathbb{Z},$ then the analytic continuation for $\lvert z\rvert>1$ of the series (\ref{eq:2f1}) is given by 
\begin{equation}
\label{eq:analform}
\begin{split}
{}_2F_1[a,b;c;z]=&\frac{\Gamma(c)\Gamma(a-b)(-z)^{-b}{}_2F_1[b,1+b-c;1+b-a;z^{-1}]}{\Gamma(a)\Gamma(c-b)}
\\
&+\frac{\Gamma(c)\Gamma(b-a)(-z)^{-a}{}_2F_1[a,1+a-c;1+a-b;z^{-1}]}{\Gamma(b)\Gamma(c-a)}
\end{split}
\end{equation} 
where $\Gamma(\cdot)$ denotes the standard gamma function.
\end{lemma}
\begin{proof}
See for instance \cite{Magnus1}, \cite{Gasper1}.
\end{proof}

\begin{lemma}
\label{lem:diff}
Suppose $b<2$, $0\leq\left|\beta-\beta_0\right|\leq1$ and $\epsilon\geq C_0$ for some $C_0>0.$ Then 
\begin{equation}
\label{eq:derseries}
\begin{split}
\frac{1}{\epsilon^b}\,\frac{d}{d\beta}\left((\beta-\beta_0)\,{}_2F_1\left[\frac{1}{q},b;1+\frac{1}{q};-\frac{C_0\left|\beta-\beta_0\right|^q}{\epsilon}\right]\right)=(\epsilon+C_0\left|\beta-\beta_0\right|^q)^{-b}
\end{split}
\end{equation} 
for all $q\in\mathbb{R}^+$ and $b\neq1/q.$
\end{lemma}
Lemma \ref{lem:diff} above is a generalization of Lemma 4.5 in \cite{Sarria1}. Its proof follows similar reasoning. Finally, the next Lemma provides us with additional tools for estimating the behaviour, as $\eta\uparrow\eta_*$, of time-dependent integrals of the type $\bar{\mathcal{K}}_i(t)$. Its proof is deferred to \S\ref{subsec:generalcase}.
\begin{lemma}
\label{lem:general}
For some $q\in\mathbb{R}^+$, suppose $u_0'(\alpha)$ satisfies (\ref{eq:expnew0}) when $\lambda\in\mathbb{R}^+$, or (\ref{eq:expnew00}) if $\lambda\in\mathbb{R}^-$. It holds:

\vspace{0.05in}

1.\, If $\lambda\in\mathbb{R}^+$ and $b>\frac{1}{q}$,
\begin{equation}
\label{eq:generalestimate}
\begin{split}
\int_0^1{\frac{d\alpha}{\mathcal{J}(\alpha,t)^b}}\sim C\mathcal{J}(\overline\alpha_i,t)^{\frac{1}{q}-b}
\end{split}
\end{equation}
for $\eta_*-\eta>0$ small and positive constants $C$ given by
\begin{equation}
\label{eq:generalconstant}
\begin{split}
C=\frac{2m\Gamma\left(1+\frac{1}{q}\right)\Gamma\left(b-\frac{1}{q}\right)}{\Gamma\left(b\right)}\left(\frac{M_0}{\left|C_1\right|}\right)^{\frac{1}{q}}.
\end{split}
\end{equation}
Here, $m\in\mathbb{N}$ denotes the finite number of locations $\overline\alpha_i$ in $[0,1]$. 

\vspace{0.05in}

2.\, If $\lambda\in\mathbb{R}^-$ and $b>\frac{1}{q}$,
\begin{equation}
\label{eq:generalestimate2}
\begin{split}
\int_0^1{\frac{d\alpha}{\mathcal{J}(\alpha,t)^b}}\sim C\mathcal{J}(\underline\alpha_j,t)^{\frac{1}{q}-b}
\end{split}
\end{equation}
for $\eta_*-\eta>0$ small and positive constants $C$ determined by 
\begin{equation}
\label{eq:generalconstant2}
\begin{split}
C=\frac{2n\Gamma\left(1+\frac{1}{q}\right)\Gamma\left(b-\frac{1}{q}\right)}{\Gamma\left(b\right)}\left(\frac{\left|m_0\right|}{C_2}\right)^{\frac{1}{q}}.
\end{split}
\end{equation}
Above, $n\in\mathbb{N}$ represents the finite number of points $\underline\alpha_j$ in $[0,1]$.

\vspace{0.05in}

3.\, Suppose $q>1/2$ and $b\in(0,1/q)$, or $q\in(0,1/2)$ and $b\in(0,2)$, satisfy $\frac{1}{q}$, $b$, $b-\frac{1}{q}\notin\mathbb{Z}$. Then for $\lambda\neq0$ and $\eta_*$ as defined in (\ref{eq:defeta*}),
\begin{equation}
\label{eq:generalestimate3}
\begin{split}
\int_0^1{\frac{d\alpha}{\mathcal{J}(\alpha,t)^b}}\sim C
\end{split}
\end{equation}
for $\eta_*-\eta>0$ small and positive constants $C$ that depend on the choice of $\lambda$, $b$ and $q$. Similarly, the integral remains bounded, and positive, for all $\eta\in[0,\eta_*]$ and $\lambda\neq0$ when $b\leq0$ and $q\in\mathbb{R}^+$.
\end{lemma}
The outline of this section is as follows. In \S\ref{subsec:lin}, we examine $L^p$, $p\in[1,+\infty]$ regularity of solutions arising from initial data satisfying (\ref{eq:expnew0}) and/or (\ref{eq:expnew00}) for $q=1$. Then, in \S\ref{subsec:generalcase} the case of arbitrary $q\in\mathbb{R}^+$ is studied. Also, regularity results concerning a class of smooth initial data larger than the one studied in \cite{Sarria1} are discussed. We remark that the case $q=1$ is considered separately from the more general argument in \S\ref{subsec:generalcase}, due to the assumptions in Lemma \ref{lem:general}.

\subsection{Global Estimates and Blow-up for $q=1$}\hfill
\label{subsec:lin}

In \cite{Sarria1}, we showed that for a particular choice of piecewise linear $u_0'(\alpha)$, a special class of solutions to the 2D Euler equations ($\lambda=1$) could develop a singularity in finite-time, whereas, for the corresponding 3D problem $(\lambda=1/2)$, solutions may converge to a nontrivial steady state as $t\to+\infty$.\footnote[2]{see Theorem \ref{thm:sarria2} in \S \ref{sec:intro}.} Therefore, it is of particular interest to determine how these results generalize to initial data satisfying (\ref{eq:expnew0}) for $q=1$. In fact, in this section we will examine $L^p$ regularity in $u_x$ for $\lambda\in\mathbb{R}$ and $p\in[1,+\infty]$.

\subsubsection{$L^{\infty}$ Regularity for $q=1$}
\label{subsubsec:linlinfty}

\begin{theorem}
\label{thm:p=1}
Consider the initial boundary value problem (\ref{eq:nonhomo})-(\ref{eq:pbc}) with $u_0'(\alpha)$ satisfying, for $q=1$, either (\ref{eq:expnew0}) when $\lambda>0$, or (\ref{eq:expnew00}) if $\lambda<0$. It holds,
\begin{enumerate}

\item\label{it:two} For $\lambda>1/2$, there exists a finite $t_*>0$ such that both the maximum $M(t)$ and the minimum $m(t)$ diverge to $+\infty$ and respectively to $-\infty$ as $t\uparrow t_*$. Moreover, for every $\alpha\notin\bigcup_{i,j}\{\overline\alpha_i\}\cup\{\underline\alpha_j\},$\, $\lim_{t\uparrow t_*}u_x(\gamma(\alpha,t),t)=-\infty$ (two-sided, everywhere blow-up).

\item\label{it:one} For $\lambda\in[0,1/2]$, solutions exist globally in time. More particularly, these vanish as $t\uparrow t_*=+\infty$ for $\lambda\in(0,1/2)$ but converge to a nontrivial steady-state if $\lambda=1/2.$

\item\label{it:three} For $\lambda<0$, there is a finite $t_*>0$ such that only the minimum diverges, $m(t)\to-\infty,$ as $t\uparrow t_*$ (one-sided, discrete blow-up). 
\end{enumerate}
\end{theorem}
\begin{proof}

Let $C$ denote a positive constant which may depend on $\lambda\neq0$.

\vspace{0.02in}
\textbf{Proofs of Statements} (\ref{it:two}) \textbf{and} (\ref{it:one})
\vspace{0.02in}

For simplicity, we prove (\ref{it:two}) and (\ref{it:one}) for the case where $M_0$ occurs at a single location $\overline\alpha\in(0,1)$\footnote[3]{The case of finitely many $\overline\alpha_i\in[0,1]$ follows similarly.}. By (\ref{eq:expnew0}), there is $0<r\leq1$ small enough such that $\epsilon+M_0-u_0^\prime(\alpha)\sim\epsilon-C_1\left|\alpha-\overline\alpha\right|$ for $0\leq\left|\alpha-\overline\alpha\right|\leq r$, $C_1<0$ and $\epsilon>0$. Then
\begin{equation}
\label{eq:app}
\begin{split}
\int_{\overline\alpha-r}^{\overline\alpha+r}{\frac{d\alpha}{(\epsilon+M_0-u_0'(\alpha))^{\frac{1}{\lambda}}}}&\sim\int_{\overline\alpha-r}^{\overline\alpha+r}{\frac{d\alpha}{(\epsilon-C_1\left|\alpha-\overline\alpha\right|)^{\frac{1}{\lambda}}}}
\\
&=\frac{2\lambda}{\left|C_1\right|(1-\lambda)}\left(\epsilon^{1-\frac{1}{\lambda}}-(\epsilon+\left|C_1\right| r)^{1-\frac{1}{\lambda}}\right)
\end{split}
\end{equation}
for $\lambda\in(0,+\infty)\backslash\{1\}$. Consequently, setting $\epsilon=\frac{1}{\lambda\eta}-M_0$ in (\ref{eq:app}) gives
\begin{equation}
\label{eq:phiest1}
\begin{split}
\bar{\mathcal{K}}_0(t)\sim
\begin{cases}
C,\,\,\,\,\,\,\,\,\,\,\,\,\,\,\,\,\,\,\,\,\,\,\,\,\,\,\,&\lambda>1,
\\
\frac{2\lambda M_0}{\left|C_1\right|(1-\lambda)}\mathcal{J}(\overline\alpha,t)^{1-\frac{1}{\lambda}},\,\,\,\,\,\,\,\,\,\,\,&\lambda\in(0,1)
\end{cases}
\end{split}
\end{equation}
for $\eta_*-\eta>0$ small, $\eta_*=\frac{1}{\lambda M_0}$ and $\mathcal{J}(\overline\alpha,t)=1-\lambda\eta(t)M_0.$ Following a similar argument, or using Lemma \ref{lem:general}(1) with $b=1+\frac{1}{\lambda}$ and $q=1$, we estimate
\begin{equation}
\label{eq:phi12}
\begin{split}
\bar{\mathcal{K}}_1(t)\sim\frac{2\lambda M_0}{\left|C_1\right|}\mathcal{J}(\overline\alpha,t)^{-\frac{1}{\lambda}}
\end{split}
\end{equation}
for any $\lambda>0$. Suppose $\lambda>1.$ Then, (\ref{eq:mainsolu}), (\ref{eq:phiest1})i) and (\ref{eq:phi12}) give
\begin{equation}
\label{eq:est21} 
\begin{split}
u_x(\gamma(\alpha,t),t)\sim C\left(\frac{1}{\mathcal{J}(\alpha,t)}-\frac{C}{\mathcal{J}(\overline\alpha,t)^{\frac{1}{\lambda}}}\right)
\end{split}
\end{equation}
for $\eta_*-\eta>0$ small. Setting $\alpha=\overline\alpha$ into (\ref{eq:est21}) and using (\ref{eq:maxmin})i) implies that
\begin{equation*}
\begin{split}
M(t)\sim\frac{C}{\mathcal{J}(\overline\alpha,t)}\rightarrow+\infty
\end{split}
\end{equation*}
as $\eta\uparrow\eta_*$. However, if $\alpha\neq\overline\alpha,$ the second term in (\ref{eq:est21}) dominates and
\begin{equation*}
\begin{split}
u_x(\gamma(\alpha,t),t)\sim-\frac{C}{\mathcal{J}(\overline\alpha,t)^{\frac{1}{\lambda}}}\rightarrow-\infty.
\end{split}
\end{equation*}
The existence of a finite $t_*>0$ for all $\lambda>1$ follows from (\ref{eq:etaivp}) and (\ref{eq:phiest1})i), which imply
\begin{equation*}
\begin{split}
t_*-t\sim C(\eta_*-\eta).
\end{split}
\end{equation*} 
Now let $\lambda\in(0,1).$ Using (\ref{eq:phiest1})ii) and (\ref{eq:phi12}) on (\ref{eq:mainsolu}), yields
\begin{equation}
\label{eq:u1}
\begin{split}
u_x(\gamma(\alpha,t),t)\sim C\left(\frac{1}{\mathcal{J}(\alpha,t)}-\frac{1-\lambda}{\mathcal{J}(\overline\alpha,t)}\right)\mathcal{J}(\overline\alpha,t)^{2(1-\lambda)}
\end{split}
\end{equation}
for $\eta_*-\eta>0$ small. Setting $\alpha=\overline\alpha$ in (\ref{eq:u1}) implies
\begin{equation}
M(t)\sim C\mathcal{J}(\overline\alpha,t)^{1-2\lambda}\to
\begin{cases}
0,\,\,\,\,\,\,\,\,\,\,&\lambda\in(0,1/2),
\\
+\infty,\,\,\,\,\,\,\,\,\,&\lambda\in(1/2,1)
\end{cases}
\end{equation}
as $\eta\uparrow \eta_*.$ If instead $\alpha\neq\overline\alpha,$
\begin{equation}
u_x(\gamma(\alpha,t),t)\sim -C\mathcal{J}(\overline\alpha,t)^{1-2\lambda}\to
\begin{cases}
0,\,\,\,\,\,\,\,\,\,\,&\lambda\in(0,1/2),
\\
-\infty,\,\,\,\,\,\,\,\,\,&\lambda\in(1/2,1)
\end{cases}
\end{equation}
as $\eta\uparrow\eta_*.$ For the threshold parameter $\lambda=1/2,$ we keep track of the constants and find that, as $\eta\uparrow\eta_*,$
\begin{equation}
u_x(\gamma(\alpha,t),t)\to
\begin{cases}
\,\,\,\,\,\frac{\left|C_1\right|}{4},\,\,\,\,\,\,\,\,\,\,\,\,\,\,&\alpha=\overline\alpha
\\
-\frac{\left|C_1\right|}{4},\,\,\,\,\,\,\,\,\,\,\,\,\,&\alpha\neq\overline\alpha.
\end{cases}
\end{equation}
Finally, (\ref{eq:etaivp}) and (\ref{eq:phiest1})ii) imply that $dt\sim C\mathcal{J}(\overline\alpha,t)^{2(\lambda-1)}d\eta$ so that
\begin{equation*}
t_*=\lim_{\eta\uparrow\eta_*}t(\eta)\sim
\begin{cases}
\frac{C}{2\lambda-1}\left(C-\lim_{\eta\uparrow\eta_*}(\eta_*-\eta)^{2\lambda-1}\right),\,\,\,\,\,\,\,\,\,&\lambda\in(0,1)\backslash\{1/2\},
\\
-C\lim_{\eta\uparrow\eta_*}\log(\eta_*-\eta),\,\,\,\,\,\,\,\,\,\,\,&\lambda=1/2.
\end{cases}
\end{equation*}
As a result, $t_*=+\infty$ if $\lambda\in(0,1/2]$ but $0<t_*<+\infty$ for $\lambda\in(1/2,1)$. Lastly, 
\begin{equation}
\begin{split}
\label{eq:l1}
\bar{\mathcal{K}}_0(t)\sim-\frac{2M_0}{\left|C_1\right|}\,\text{log}(\eta_*-\eta)
\end{split}
\end{equation}
for $0<\eta_*-\eta<<1$ small and $\lambda=1$. Then, two-sided, everywhere blow-up in finite-time follows just as above from (\ref{eq:mainsolu}), (\ref{eq:etaivp}), (\ref{eq:phi12}) and (\ref{eq:l1}). Finally, the case $\lambda=0$ follows from the results in \cite{Sarria1}.

\vspace{0.02in}
\textbf{Proof of Statement} (\ref{it:three})
\vspace{0.02in}

For $\lambda<0$, set $\eta_*=\frac{1}{\lambda m_0}$. Then $\bar{\mathcal{K}}_0(t)$ remains finite, and positive, for all $\eta\in[0,\eta_*]$. In fact, one can easily show that 
\begin{equation}
\label{eq:quick3p}
\begin{split}
1\leq\bar{\mathcal{K}}_0(t)\leq\left(1+\frac{M_0}{\left|m_0\right|}\right)^{\frac{1}{\left|\lambda\right|}}
\end{split}
\end{equation}
if $\lambda\in[-1,0)$, while 
\begin{equation}
\label{eq:newestq}
\begin{split}
0<\int_0^1{\left(1+\frac{u_0'(\alpha)}{\left|m_0\right|}\right)^{\frac{1}{\left|\lambda\right|}}d\alpha}\leq\bar{\mathcal{K}}_0(t)\leq1
\end{split}
\end{equation} 
for $\lambda<-1$. Similarly, when $\lambda\in[-1,0)$ and $\eta\in[0,\eta_*]$, 
\begin{equation}
\label{eq:quick4p}
\begin{split}
1\leq\bar{\mathcal{K}}_1(t)\leq\left(\frac{\left|m_0\right|}{M_0+\left|m_0\right|}\right)^{1+\frac{1}{\lambda}}.
\end{split}
\end{equation}
However, if $\lambda<-1,$ we need to estimate $\bar{\mathcal{K}}_1(t)$ for $\eta_*-\eta>0$ small. To do so, we proceed analogously to the derivation of (\ref{eq:phiest1}). For simplicity, assume $u_0'(\alpha)$ achieves its least value $m_0<0$ at a single point $\underline\alpha\in(0,1)$. Then (\ref{eq:expnew00}) with $q=1$ implies that $u_0'(\alpha)\sim m_0+C_2\left|\alpha-\underline\alpha\right|$ for $0\leq\left|\alpha-\underline\alpha\right|\leq s$, $C_2>0$ and $0<s\leq1$. It follows that
\begin{equation}
\begin{split}
\label{eq:l-}
\int_{\underline\alpha-s}^{\underline\alpha+s}{\frac{d\alpha}{(\epsilon+u_0'(\alpha)-m_0)^{1+\frac{1}{\lambda}}}}&\sim\int_{\underline\alpha-s}^{\underline\alpha+s}{\frac{d\alpha}{(\epsilon+C_2\left|\alpha-\underline\alpha\right|)^{1+\frac{1}{\lambda}}}}
\\
&= \frac{2\left|\lambda\right|}{C_2}\left((\epsilon+C_2s)^{\frac{1}{\left|\lambda\right|}}-\epsilon^{\frac{1}{\left|\lambda\right|}}\right)
\end{split}
\end{equation}
for $\epsilon>0$. By substituting $\epsilon=m_0-\frac{1}{\lambda\eta}$ into (\ref{eq:l-}), we find that $\bar{\mathcal{K}}_1(t)$ has a finite, positive limit as $\eta\uparrow\eta_*$ for $\lambda<-1$. This implies that for $\lambda<0$, both time-dependent integrals in (\ref{eq:mainsolu}) remain bounded and positive for all $\eta\in[0,\eta_*]$. Consequently, blow-up of (\ref{eq:mainsolu}), as $\eta\uparrow\eta_*$, will follow from the space-dependent term, $\mathcal{J}(\alpha,t)^{-1}$, evaluated at $\alpha=\underline\alpha$. In this way, we set $\alpha=\underline\alpha$ into (\ref{eq:mainsolu}) and use (\ref{eq:maxmin})ii) to obtain
$$m(t)\sim\frac{Cm_0}{\mathcal{J}(\overline\alpha,t)}\to-\infty$$
as $\eta\uparrow\eta_*$. In contrast, for $\alpha\neq\underline\alpha,$ the definition of $m_0$ implies that the space-dependent term now remains bounded for $\eta\in[0,\eta_*]$. Finally, the existence of a finite blow-up time $t_*>0$ for the minimum follows from (\ref{eq:etaivp}) and the estimates on $\bar{\mathcal{K}}_0(t)$. In fact, by (\ref{eq:etaivp}), $t_*=\eta_*$ for $\lambda=-1$, while (\cite{Sarria1})
\begin{equation}
\begin{cases}
\label{eq:timebounds}
\eta_*\leq t_*<+\infty,\,\,\,&\lambda<-1,
\\
\eta_*\left(1-\frac{M_0}{m_0}\right)^{-2}\leq t_*\leq\eta_*,\,\,\,\,\,&\lambda\in(-1,0).
\end{cases}
\end{equation}
See \S\ref{sec:examples} for examples.\end{proof}

In preparation for the next section, we recall some formulas, as well as upper and lower bounds, derived in \cite{Sarria1} for the $L^p$ norm of $u_x$. For as long as a solution exists, (\ref{eq:sum}) and (\ref{eq:mainsolu}) imply that
\begin{equation*}
\begin{split}
\left\|u_x(\cdot,t)\right\|_p^p=\frac{1}{\left|\lambda\eta(t)\right|^p\bar{\mathcal{K}}_0(t)^{^{1+2\lambda p}}}\int_0^1{\left|\frac{1}{\mathcal{J}(\alpha,t)^{^{1+\frac{1}{\lambda p }}}}-\frac{\bar{\mathcal{K}}_1(t)}{\bar{\mathcal{K}}_0(t)\mathcal{J}(\alpha,t)^{\frac{1}{\lambda p}}}\right|^pd\alpha}
\end{split}
\end{equation*}
for $\lambda\neq 0$ and $p\in[1,+\infty)$. Using the above and some standard inequalities yields
\begin{equation}
\label{eq:upper}
\begin{split}
\left\|u_x(\cdot,t)\right\|_p^p\leq\frac{2^{p-1}}{\left|\lambda\eta(t)\right|^p\bar{\mathcal{K}}_0(t)^{^{1+2\lambda p}}}\left(\int_0^1{\frac{d\alpha}{\mathcal{J}(\alpha,t)^{^{p+\frac{1}{\lambda }}}}}+\frac{\bar{\mathcal{K}}_1(t)^{^p}}{\bar{\mathcal{K}}_0(t)^{^{p-1}}}\right)
\end{split}
\end{equation}
and
\begin{equation}
\label{eq:lower}
\begin{split}
\left\|u_x(\cdot,t)\right\|_p\geq\frac{1}{\left|\lambda\eta(t)\right|\bar{\mathcal{K}}_0(t)^{^{2\lambda+\frac{1}{p}}}}\left|\int_0^1{\frac{d\alpha}{\mathcal{J}(\alpha,t)^{^{1+\frac{1}{\lambda p }}}}}-\frac{\bar{\mathcal{K}}_1(t)}{\bar{\mathcal{K}}_0(t)}\int_0^1{\frac{d\alpha}{\mathcal{J}(\alpha,t)^{\frac{1}{\lambda p}}}}\right|.
\end{split}
\end{equation}
Moreover, the energy function $E(t)\equiv\left\|u_x(\cdot,t)\right\|_2^2$ is explicitly given by
\begin{equation}
\label{eq:energy}
\begin{split}
E(t)=\left(\lambda\eta(t)\bar{\mathcal{K}}_0(t)^{1+2\lambda}\right)^{-2}\left(\bar{\mathcal{K}}_0(t)\bar{\mathcal{K}}_2(t)-\bar{\mathcal{K}}_1(t)^2\right).
\end{split}
\end{equation}
Lastly, multiplying (\ref{eq:nonhomo})i) by $u_x$, integrating by parts, and using (\ref{eq:pbc}), (\ref{eq:sum}) and (\ref{eq:mainsolu}), gives
\begin{equation}
\label{eq:derenergy}
\begin{split}
\dot E(t)&=(1+2\lambda)\int_0^1{u_x(x,t)^3dx}=(1+2\lambda)\int_0^1{u_x(\gamma(\alpha,t),t)^3\gamma_\alpha(\alpha,t)\,d\alpha}
\\
&=\frac{1+2\lambda}{(\lambda\eta(t))^3}
\left[\frac{\bar{\mathcal{K}}_3(t)}{\bar{\mathcal{K}}_1(t)}-\frac{3\bar{\mathcal{K}}_2(t)}{\bar{\mathcal{K}}_0(t)}+2\left(\frac{\bar{\mathcal{K}}_1(t)}{\bar{\mathcal{K}}_0(t)}\right)^2\right]\frac{\bar{\mathcal{K}}_1(t)}{\bar{\mathcal{K}}_0(t)^{1+6\lambda}}.
\end{split}
\end{equation}
The reader may refer to \cite{Sarria1} for details on the above.

\subsubsection{Further $L^p$ Regularity for $\lambda\neq0$, $q=1$ and $p\in[1,+\infty)$}\hfill
\label{subsubsec:linlp}

In the previous section, we established the existence of a finite $t_*>0$ such that $\left\|u_x\right\|_{\infty}$ diverges as $t\uparrow t_*$ for all $\lambda\in\mathbb{R}\backslash[0,1/2]$ and initial data satisfying (\ref{eq:expnew0}) and/or (\ref{eq:expnew00}) for $q=1$ relative to the sign of $\lambda$. If instead, $\lambda\in[0,1/2]$, we proved that solutions remain in $L^{\infty}$ for all time. In this section, we examine further $L^p$ regularity of $u_x$, as $t\uparrow t_*$, for $\lambda\in\mathbb{R}\backslash[0,1/2]$ and $p\in[1,+\infty)$. 

\begin{theorem}
\label{thm:q=1,p>=1}
For the initial boundary value problem (\ref{eq:nonhomo})-(\ref{eq:pbc}), let $t_*>0$ denote the finite $L^{\infty}$ blow-up time for $u_x$ in Theorem \ref{thm:p=1}. Further, for $q=1$, suppose $u_0'(\alpha)$ satisfies (\ref{eq:expnew0}) when $\lambda>0$, or (\ref{eq:expnew00}) if $\lambda<0$.
\begin{enumerate}

\item\label{it:onep} For $\lambda>1/2$ and $p>1$, $\lim_{t\uparrow t_*}\left\|u_x\right\|_p=+\infty$.

\item\label{it:twop} For $\lambda<0$ and $t\in[0,t_*]$, $u_x$ remains integrable; moreover, if $\frac{1}{1-p}<\lambda<0$ and $p>1$, then $u_x\in L^p$ for all $t\in[0,t_*]$. 

\item\label{it:threep} The energy $E(t)=\left\|u_x\right\|_2^2$ diverges if $\lambda\in(-\infty,-1]\cup(1/2,+\infty)$ as $t\uparrow t_*$ but remains finite for $t\in[0,t_*]$ if $\lambda\in(-1,0)$. Also, $\lim_{t\uparrow t_*}\dot E(t)=+\infty$ when $\lambda\in(-\infty,-1/2)\cup(1/2,+\infty)$, whereas, $\dot E(t)\equiv0$ if $\lambda=-1/2$ while $\dot E(t)$ stays bounded for $t\in[0,t_*]$ if $\lambda\in(-1/2,0)$.

\end{enumerate}
\end{theorem}


\begin{proof}
Let $C$ denote a positive constant that may depend on the choice of $\lambda$ and $p\in[1,+\infty)$.

\vspace{0.02in}
\textbf{Proof of Statement} (\ref{it:onep})
\vspace{0.02in}

First, suppose $\lambda>0$ and set $\eta_*=\frac{1}{\lambda M_0}$. For simplicity, we prove part (\ref{it:onep}) under the assumption that $M_0>0$ occurs at a single point $\overline\alpha\in(0,1)$. Using Lemma \ref{lem:general}(1) with $b=1+\frac{1}{\lambda p}$, $q=1$ and $p\geq1$, yields

\begin{equation}
\label{eq:p13}
\begin{split}
\int_0^1{\frac{d\alpha}{\mathcal{J}(\alpha,t)^{1+\frac{1}{\lambda p}}}}\sim\frac{2\lambda pM_0}{\left|C_1\right|}\mathcal{J}(\overline\alpha,t)^{-\frac{1}{\lambda p}}
\end{split}
\end{equation}
for $\eta_*-\eta>0$ small. Similarly, taking $b=p+\frac{1}{\lambda}$ we find that
\begin{equation}
\label{eq:p141}
\begin{split}
\int_0^1{\frac{d\alpha}{\mathcal{J}(\alpha,t)^{p+\frac{1}{\lambda}}}}\sim \frac{2\lambda M_0}{\left|C_1\right|(\lambda(p-1)+1)}\mathcal{J}(\overline\alpha,t)^{1-p-\frac{1}{\lambda}}.
\end{split}
\end{equation}
Moreover, following the argument that led to estimate (\ref{eq:phiest1}), with $\frac{1}{\lambda p}$ instead of $\frac{1}{\lambda}$, gives
\begin{equation}
\label{eq:p14}
\int_0^1{\frac{d\alpha}{\mathcal{J}(\alpha,t)^{\frac{1}{\lambda p}}}}\sim
\begin{cases}
\frac{2\lambda pM_0}{\left|C_1\right|(1-\lambda p)}\mathcal{J}(\overline\alpha,t)^{1-\frac{1}{\lambda p}},\,\,&\lambda\in(0,1/p),
\\
C,\,\,\,\,&\lambda>1/p
\end{cases}
\end{equation}
for $p\geq1$ and $\eta_*-\eta>0$ small. 
Suppose $\lambda, p>1$ so that $\lambda>1/p$. Then, using (\ref{eq:phiest1})i), (\ref{eq:phi12}), (\ref{eq:p13}) and (\ref{eq:p14})ii) in (\ref{eq:lower}), implies that 
\begin{equation*}
\begin{split}
\left\|u_x(\cdot,t)\right\|_p&\geq\frac{1}{\left|\lambda\eta(t)\right|\bar{\mathcal{K}}_0(t)^{^{2\lambda+\frac{1}{p}}}}\left|\int_0^1{\frac{d\alpha}{\mathcal{J}(\alpha,t)^{^{1+\frac{1}{\lambda p }}}}}-\frac{\bar{\mathcal{K}}_1(t)}{\bar{\mathcal{K}}_0(t)}\int_0^1{\frac{d\alpha}{\mathcal{J}(\alpha,t)^{\frac{1}{\lambda p}}}}\right|
\\
&\sim C\left|C\mathcal{J}(\overline\alpha,t)^{-\frac{1}{\lambda p}}-\mathcal{J}(\overline\alpha,t)^{-\frac{1}{\lambda}}\right|
\\
&\sim C\mathcal{J}(\overline\alpha,t)^{-\frac{1}{\lambda}}\to+\infty
\end{split}
\end{equation*}
as $\eta\uparrow\eta_*$. Next, let $p\in(1,2)$ and $\lambda\in(1/2,1/p)\subset(1/2,1)$. Then, using (\ref{eq:phiest1})ii), (\ref{eq:phi12}), (\ref{eq:p13}) and (\ref{eq:p14})i) in (\ref{eq:lower}), gives
\begin{equation*}
\begin{split}
\left\|u_x(\cdot,t)\right\|_p&\geq\frac{1}{\left|\lambda\eta(t)\right|\bar{\mathcal{K}}_0(t)^{^{2\lambda+\frac{1}{p}}}}\left|\int_0^1{\frac{d\alpha}{\mathcal{J}(\alpha,t)^{^{1+\frac{1}{\lambda p }}}}}-\frac{\bar{\mathcal{K}}_1(t)}{\bar{\mathcal{K}}_0(t)}\int_0^1{\frac{d\alpha}{\mathcal{J}(\alpha,t)^{\frac{1}{\lambda p}}}}\right|
\\
&\sim C\left|1-\frac{1-\lambda}{1-\lambda p}\right|\mathcal{J}(\overline\alpha,t)^{\rho(\lambda,p)}
\\
&=C\mathcal{J}(\overline\alpha,t)^{\rho(\lambda,p)}
\end{split}
\end{equation*}
for $\eta_*-\eta>0$ small and $\rho(\lambda,p)=2(1-\lambda)-\frac{1}{p}$. However, for $\lambda$ and $p$ as prescribed, we see that $\rho(\lambda,p)<0$ for $1-\frac{1}{2p}<\lambda<\frac{1}{p}$ and $p\in(1,3/2)$. Therefore, for any $\lambda\in(1/2,1)$ there is $1-p>0$ arbitrarily small such that $\left\|u_x\right\|_p\to+\infty$ as $\eta\uparrow\eta_*$. Finally, if $\lambda=1$ we have $\lambda>1/p$ for $p>1$, as a result, (\ref{eq:phi12}), (\ref{eq:l1}), (\ref{eq:p13}) and (\ref{eq:p14})iii) imply that
\begin{equation*}
\begin{split}
\left\|u_x(\cdot,t)\right\|_p&\geq\frac{1}{\left|\lambda\eta(t)\right|\bar{\mathcal{K}}_0(t)^{^{2\lambda+\frac{1}{p}}}}\left|\int_0^1{\frac{d\alpha}{\mathcal{J}(\alpha,t)^{^{1+\frac{1}{\lambda p }}}}}-\frac{\bar{\mathcal{K}}_1(t)}{\bar{\mathcal{K}}_0(t)}\int_0^1{\frac{d\alpha}{\mathcal{J}(\alpha,t)^{\frac{1}{\lambda p}}}}\right|
\\
&\sim C\mathcal{J}(\overline\alpha,t)^{-1}(-\log(\eta_*-\eta))^{-3-\frac{1}{p}}
\end{split}
\end{equation*}
for $0<\eta_*-\eta<<1$ small, and so, $\left\|u_x\right\|_p\to+\infty$ as $\eta\uparrow\eta_*$. The existence of a finite blow-up time $t_*>0$ follows from Theorem \ref{thm:p=1}. 

\vspace{0.02in}
\textbf{Proof of Statement} (\ref{it:twop})
\vspace{0.02in}

Suppose $\lambda<0$ and set $\eta_*=\frac{1}{\lambda m_0}$. First, recall from the proof of Theorem \ref{thm:p=1} that $\bar{\mathcal{K}}_i(t),\, i=0,1$ remain finite and positive for all $\eta\in[0,\eta_*]$. Furthermore, in Theorem \ref{thm:p=1} we established the existence of a finite blow-up time $t_*>0$ for the minimum $m(t)$. Consequently, the upper bound (\ref{eq:upper}) implies that
\begin{equation}
\label{eq:ref2}
\begin{split}
\lim_{t\uparrow t_*}\left\|u_x(\cdot,t)\right\|_p<+\infty\,\,\,\,\,\,\,\Leftrightarrow\,\,\,\,\,\,\,\lim_{t\uparrow t_*}\int_0^1{\frac{d\alpha}{\mathcal{J}(\alpha,t)^{p+\frac{1}{\lambda}}}}<+\infty
\end{split}
\end{equation}
for $\lambda<0$ and $p\geq1$. However, if $p=1$, (\ref{eq:ref2})ii) is just $\bar{\mathcal{K}}_1(t)$, which remains finite as $t\uparrow t_*$. As a result, $u_x\in L^1$ for all $t\in[0,t_*]$ and $\lambda<0$.
If $p>1$, we recreate the argument in (\ref{eq:l-}), with $p+\frac{1}{\lambda}$ instead of $1+\frac{1}{\lambda}$, and find that for
$\frac{1}{1-p}<\lambda<0$ and $p>1$, 
the integral remains finite and positive as $\eta\uparrow\eta_*$. Consequently, (\ref{eq:ref2}) implies that 
$$\lim_{t\uparrow t_*}\left\|u_x(\cdot,t)\right\|_p<+\infty$$
for all $\frac{1}{1-p}<\lambda<0$ and $p>1$. We remark that the lower bound (\ref{eq:lower}) yields no information regarding $L^p$ blow-up of $u_x$, as $t\uparrow t_*$, for parameter values $-\infty<\lambda<\frac{1}{1-p}$, $p>1$. Nonetheless, we can use (\ref{eq:energy}) and (\ref{eq:derenergy}) to obtain additional blow-up information on energy-related quantities. 

\vspace{0.02in}
\textbf{Proof of Statement} (\ref{it:threep})
\vspace{0.02in}

For $\lambda>1/2$, blow-up of $E(t)$ and $\dot E(t)$ to $+\infty$ as $t\uparrow t_*$ is a consequence of part (\ref{it:onep}) above. Further, setting $p=2$ in part (\ref{it:twop}) implies that $E(t)$ remains bounded for all $\lambda\in(-1,0)$ and $t\in[0,t_*]$. Now, (\ref{eq:derenergy})i) yields
\begin{equation}
\label{eq:l3} 
\left|\dot E(t)\right|\leq\left|1+2\lambda\right|\left\|u_x(\cdot,t)\right\|_3^3,
\end{equation}
and so setting $p=3$ in part (\ref{it:twop}) implies that $\dot E(t)$ remains finite for $\lambda\in[-1/2,0)$ and $t\in[0,t_*]$. According to these results, we have yet to determine the behaviour of $E(t)$ as $t\uparrow t_*$ for $\lambda\leq-1$ and $\dot E(t)$ when $\lambda<-1/2$. To do so, we will use formulas (\ref{eq:energy}) and (\ref{eq:derenergy}). From Lemma \ref{lem:general}(2) with $b=3+\frac{1}{\lambda}$, $q=1$ and $\lambda<-1/2$, we find that 
\begin{equation}
\label{eq:k3}
\bar{\mathcal{K}}_3(t)\sim \frac{2\lambda\left|m_0\right|}{C_2(1+2\lambda)}\mathcal{J}(\underline\alpha,t)^{-2-\frac{1}{\lambda}}
\end{equation}
for $\eta_*-\eta>0$ small. Also, following the argument in (\ref{eq:l-}), with $2+\frac{1}{\lambda}$ instead of $1+\frac{1}{\lambda}$, we derive 
\begin{equation}
\label{eq:k2}
\bar{\mathcal{K}}_2(t)\sim
\begin{cases}
\frac{2\lambda\left|m_0\right|}{C_2(1+\lambda)}\mathcal{J}(\underline\alpha,t)^{-1-\frac{1}{\lambda}},\,\,\,\,&\lambda<-1,
\\
-C\log(\eta_*-\eta),\,\,\,\,&\lambda=-1,
\\
C,\,\,\,\,\,\,\,&\lambda\in(-1,0).
\end{cases}
\end{equation}
Since both $\bar{\mathcal{K}}_i(t)$, $i=0,1$ stay finite and positive for all $\eta\in[0,\eta_*]$ and $\lambda<0$, (\ref{eq:energy}) tells us that blow-up in $\bar{\mathcal{K}}_2(t)$ leads to a diverging $E(t)$. Then, (\ref{eq:k2})i) implies that for $\lambda<-1$,
$$E(t)\sim C\mathcal{J}(\underline\alpha,t)^{-1-\frac{1}{\lambda}}\to+\infty$$
as $\eta\uparrow\eta_*$. Similarly for $\lambda=-1$ by using (\ref{eq:k2})ii) instead. Clearly, this also implies blow-up of $\dot E(t)$ to $+\infty$ as $t\uparrow t_*$ for all $\lambda\leq-1$. Finally, from (\ref{eq:derenergy})ii), (\ref{eq:k3}) and (\ref{eq:k2})iii), 
$$\dot E(t)\sim\frac{Cm_0^3(1+2\lambda)}{\mathcal{J}(\underline\alpha,t)^{2+\frac{1}{\lambda}}}\to+\infty$$
as $\eta\uparrow\eta_*$ for all $\lambda\in(-1,-1/2)$. The existence of a finite $t_*>0$ follows from Theorem \ref{thm:p=1}(\ref{it:three}).\end{proof}

From the results established thus far, we are able to obtain a complete description of the $L^3$ regularity for $u_x$: if $\lambda\in[0,1/2]$, $\lim_{t\to +\infty}\left\|u_x\right\|_3=C$ where $C\in\mathbb{R}^+$ for $\lambda=1/2$ but $C=0$ if $\lambda\in(0,1/2)$, while, for $t_*>0$ the finite $L^{\infty}$ blow-up time for $u_x$ in Theorem \ref{thm:p=1},
\begin{equation}
\label{eq:l3q=1}
\lim_{t\uparrow t_*}\left\|u_x(\cdot,t)\right\|_3=
\begin{cases}
+\infty,\,\,\,\,\,\,\,&\lambda\in(-\infty,-1/2]\cup(1/2,+\infty),
\\
C\in\mathbb{R}^+,\,\,\,\,&\lambda\in(-1/2,0).
\end{cases}
\end{equation}

\begin{remark}
For $t_*>0$ the finite $L^{\infty}$ blow-up time for $u_x$ in Theorem \ref{thm:p=1}, we may use (\ref{eq:derenergy}), (\ref{eq:k3}) and (\ref{eq:k2}), as well as Theorem \ref{thm:q=1,p>=1}, to establish a global bound on $\int_0^1{u_x^3dx}$ if $\lambda\in[0,1/2]$, or for $t\in[0,t_*]$ when $\lambda\in(-1/2,0)$, whereas
\begin{equation}
\label{eq:3integral}
\lim_{t\uparrow t_*}\int_0^1{u_x(x,t)^3dx}=
\begin{cases}
+\infty,\,\,\,\,\,\,\,&\lambda>1/2,
\\
-\infty,\,\,\,\,&\lambda\leq-1/2.
\end{cases}
\end{equation}
We also note that, unlike the result in Theorem \ref{thm:lpintro}(\ref{it:ener}) of \S\ref{sec:intro}, (\ref{eq:3integral}) and the change in sign through $\lambda=-1/2$ of the term $1+2\lambda$ in (\ref{eq:derenergy}), prevent the possibility of blow-up of $\dot E(t)$ towards $-\infty$, which might otherwise have played a role in the study of weak solutions from the point of view of energy dissipation.
\end{remark}

\begin{remark}
Notice that the two-sided, everywhere blow-up found in Theorem \ref{thm:p=1} for $\lambda>1/2$ corresponds, in Theorem \ref{thm:q=1,p>=1}, to $L^p$ blow-up of $u_x$ for any $p>1$. On the other hand, $u_x$ remains integrable for all $\lambda<0$ and $t\in[0,t_*]$ but, as $t\uparrow t_*$, undergoes an $L^{\infty}$ blow-up of the one-sided, discrete type for $\lambda<0$. Then, as the magnitude of $\lambda<0$ decreases, $u_x$ is guaranteed to remain, for $t\in[0,t_*]$, in smaller $L^p$ spaces with $p\in(1,+\infty)$. In the coming sections, we will find that a similar correspondence between the ``strengths'' of the $L^{\infty}$ and $L^p$, $p\in[1,+\infty)$, blow-up in $u_x$, as $t\uparrow t_*$, also holds for other $q>0$. 
\end{remark}

\subsection{Global Estimates and Blow-up for $\lambda\in\mathbb{R}$ and $q>0$}\hfill
\label{subsec:generalcase}

In this section, we study the case of arbitrary $q>0$. As in the previous sections, $L^{p}$ regularity of $u_x$ for $\lambda\in\mathbb{R}$ and $p\in[1,+\infty]$ is examined. In addition, the behaviour of the jacobian (\ref{eq:sum}) is considered. Particularly, we will show that if $q\geq1$, no blow-up occurs in stagnation point-form solutions to the 3D incompressible Euler equations, whereas, for the corresponding 2D case, no spontaneous singularity forms when $q\geq2$. Finally, a class of smooth, periodic initial data larger than the one considered in \cite{Sarria1} is studied. 
Before stating and proving our results, we first establish Lemma \ref{lem:general} and obtain estimates on $\bar{\mathcal{K}}_0(t)$ and $\bar{\mathcal{K}}_1(t)$. 

\vspace{0.02in}
\textbf{Proof of Lemma} \ref{lem:general}(1)
\vspace{0.02in}

For simplicity, we prove statement (1) for functions $u_0'$ that attain their greatest value $M_0>0$ at a single location $\overline\alpha\in(0,1)$. The case of several $\overline\alpha_i\in[0,1]$ follows similarly. From (\ref{eq:expnew0}), there is $0<r\leq1$ such that $\epsilon+M_0-u_0'(\alpha)\sim\epsilon-C_1\left|\alpha-\overline\alpha\right|^q$ for $q\in\mathbb{R}^+$, $\epsilon>0$ and $0\leq\left|\alpha-\overline\alpha\right|\leq r$. Therefore 
\begin{equation*}
\begin{split}
&\int_{\overline\alpha-r}^{\overline\alpha+r}{\frac{d\alpha}{(\epsilon+M_0-u_0'(\alpha))^b}}\sim\int_{\overline\alpha-r}^{\overline\alpha+r}{\frac{d\alpha}{(\epsilon-C_1\left|\alpha-\overline\alpha\right|^q)^b}}
\\
&=\epsilon^{-b}\left[\int_{\overline\alpha-r}^{\overline\alpha}{\left(1+\frac{\left|C_1\right|}{\epsilon}\left(\overline\alpha-\alpha\right)^q\right)^{-b}d\alpha}+\int_{\overline\alpha}^{\overline\alpha+r}{\left(1+\frac{\left|C_1\right|}{\epsilon}\left(\alpha-\overline\alpha\right)^q\right)^{-b}d\alpha}\right]
\end{split}
\end{equation*}
for $b\in\mathbb{R}$. Making the change of variables 
$$\sqrt{\frac{\left|C_1\right|}{\epsilon}}(\overline\alpha-\alpha)^{\frac{q}{2}}=\tan\theta,\,\,\,\,\,\,\,\,\,\,\,\,\,\,\sqrt{\frac{\left|C_1\right|}{\epsilon}}(\alpha-\overline\alpha)^{\frac{q}{2}}=\tan\theta$$
in the first and second integrals inside the bracket, respectively, we find that
\begin{equation}
\label{eq:general2}
\begin{split}
&\int_{\overline\alpha-r}^{\overline\alpha+r}{\frac{d\alpha}{(\epsilon+M_0-u_0'(\alpha))^b}}\sim\frac{4}{q\left|C_1\right|^{\frac{1}{q}}\epsilon^{b-\frac{1}{q}}}\int_0^{\frac{\pi}{2}}{\frac{(\cos\theta)^{^{2b-\frac{2}{q}-1}}}{(\sin\theta)^{^{1-\frac{2}{q}}}}d\theta}
\end{split}
\end{equation}
for small $\epsilon>0$. Suppose $b>\frac{1}{q}$, then setting $\epsilon=\frac{1}{\lambda\eta}-M_0$ in (\ref{eq:general2}) implies
\begin{equation}
\label{eq:general3}
\begin{split}
\int_{0}^{1}{\frac{d\alpha}{\mathcal{J}(\alpha,t)^b}}\sim\frac{C}{\mathcal{J}(\overline\alpha,t)^{^{b-\frac{1}{q}}}}
\end{split}
\end{equation}
for $\eta_*-\eta>0$ small, $\eta_*=\frac{1}{\lambda M_0}$ and 
\begin{equation}
\label{eq:generalcst}
\begin{split}
C=\frac{4}{q}\left(\frac{M_0}{\left|C_1\right|}\right)^{\frac{1}{q}}\int_0^{\frac{\pi}{2}}{\frac{(\cos\theta)^{^{2b-\frac{2}{q}-1}}}{(\sin\theta)^{^{1-\frac{2}{q}}}}d\theta}.
\end{split}
\end{equation}
Now, since the beta function satisfies (see for instance \cite{Gamelin1}): 
\begin{equation}
\label{eq:gammarel}
\begin{split}
B(p,s)=\int_0^1{t^{p-1}(1-t)^{s-1}dt}=\frac{\Gamma(p)\Gamma(s)}{\Gamma(p+s)},\,\,\,\,\,\,\,\,\,\,\,\,\,\Gamma(1+y)=y\Gamma(y)
\end{split}
\end{equation}
for $p,s,y>0$, then, letting $t=\sin^2\theta$, $p=\frac{1}{q}$ and $s=b-\frac{1}{q}$ into (\ref{eq:gammarel})i), and using (\ref{eq:gammarel})ii), one has
\begin{equation}
\label{eq:gammarel2}
\begin{split}
2\int_0^{\frac{\pi}{2}}{\frac{(\cos\theta)^{^{2b-\frac{2}{q}-1}}}{(\sin\theta)^{^{1-\frac{2}{q}}}}d\theta}=\frac{q\,\Gamma\left(1+\frac{1}{q}\right)\Gamma\left(b-\frac{1}{q}\right)}{\Gamma(b)},\,\,\,\,\,\,\,\,\,\,\,\,b>\frac{1}{q}.
\end{split}
\end{equation}
The result follows from (\ref{eq:general3}), (\ref{eq:generalcst}) and (\ref{eq:gammarel2}).

\vspace{0.02in}
\textbf{Proof of Lemma} \ref{lem:general}(2)
\vspace{0.02in}

Follows from an analogous argument using (\ref{eq:expnew00}) and $\eta_*=\frac{1}{\lambda m_0}$ instead. 

\vspace{0.02in}
\textbf{Proof of Lemma} \ref{lem:general}(3)
\vspace{0.02in}

The last claim in (3) follows trivially if $b\leq0$ and $q\in\mathbb{R}^+$ due to the ``almost everywhere'' continuity and boundedness of $u_0'$. To establish the remaining claims, we make use of Lemmas \ref{lem:analcont} and \ref{lem:diff}. However, in order to use the latter, we require that $b\in(0,2)$ and $b\neq1/q$. Since the case $b>1/q$ was established in parts (1) and (2) above, suppose that $b\in(0,1/q)$ and $b\in(0,2)$, or equivalently $q>1/2$ and $b\in(0,1/q)$, or $q\in(0,1/2)$ and $b\in(0,2)$. First, for $q$ and $b$ as prescribed, consider $\lambda>0$ and, for simplicity, assume $M_0$ occurs at a single point $\overline\alpha\in(0,1)$. Then, (\ref{eq:expnew0}) and Lemma \ref{lem:diff} imply that
\begin{equation}
\label{eq:lastgen1}
\begin{split}
\int_{\overline\alpha-r}^{\overline\alpha+r}{\frac{d\alpha}{(\epsilon+M_0-u_0'(\alpha))^b}}&\sim\int_{\overline\alpha-r}^{\overline\alpha+r}{\frac{d\alpha}{(\epsilon-C_1\left|\alpha-\overline\alpha\right|^q)^b}}
\\
&=2r\epsilon^{-b}\,{}_2F_1\left[\frac{1}{q},b,1+\frac{1}{q},\frac{C_1r^q}{\epsilon}\right]
\end{split}
\end{equation}
for $\epsilon\geq\left|C_1\right|\geq\left|C_1\right|r^q>0$ and $0\leq\left|\alpha-\overline\alpha\right|\leq r$. Now, the restriction on $\epsilon$ implies that $-1\leq\frac{C_1r^q}{\epsilon}<0$. However, our ultimate goal is to let $\epsilon$ vanish, so that, eventually, the argument $\frac{C_1r^q}{\epsilon}$ of the series in (\ref{eq:lastgen1})ii) will leave the unit circle, particularly $\frac{C_1r^q}{\epsilon}<-1$. At that point, definition \ref{eq:2f1} for the series no longer holds and we turn to its analytic continuation in Lemma \ref{lem:analcont}. Accordingly, taking $\epsilon>0$ small enough such that $\left|C_1\right|r^q>\epsilon>0$, we apply Lemma \ref{lem:analcont} to (\ref{eq:lastgen1}) and obtain
\begin{equation}
\label{eq:lastgen2}
\begin{split}
\frac{2r}{\epsilon^{b}}\,{}_2F_1\left[\frac{1}{q},b,1+\frac{1}{q},\frac{C_1r^q}{\epsilon}\right]=\frac{2r^{1-qb}}{(1-bq)\left|C_1\right|^b}+\frac{2\Gamma\left(1+\frac{1}{q}\right)\Gamma\left(b-\frac{1}{q}\right)}{\Gamma(b)\left|C_1\right|^{\frac{1}{q}}\epsilon^{b-\frac{1}{q}}}+\psi(\epsilon)
\end{split}
\end{equation}
for $\psi(\epsilon)=o(1)$ as $\epsilon\to0$, and either $q>1/2$ and $b\in(0,1/q)$, or $q\in(0,1/2)$ and $b\in(0,2)$. In addition, due to the assumptions in Lemma \ref{lem:analcont} we require that $\frac{1}{q}$, $b$, $b-\frac{1}{q}\notin\mathbb{Z}$. Finally, since $b-\frac{1}{q}<0$, substituting $\epsilon=\frac{1}{\lambda\eta}-M_0$ into (\ref{eq:lastgen1}) and (\ref{eq:lastgen2}), implies that 
\begin{equation}
\label{eq:lastgen3}
\begin{split}
\int_{0}^{1}{\frac{d\alpha}{\mathcal{J}(\alpha,t)^b}}\sim C
\end{split}
\end{equation}
for $\eta_*-\eta>0$ small, $\eta_*=\frac{1}{\lambda M_0}$, and positive constants $C$ that depend on $\lambda>0$, $b$ and $q$. An analogous argument follows for $\lambda<0$ by using (\ref{eq:expnew00}) instead of (\ref{eq:expnew0}).\hfill$\square$

Using Lemma \ref{lem:general}, we now derive estimates for $\bar{\mathcal{K}}_i(t)$, $i=0,1$, which will be used in subsequent regularity Theorems.

\subsubsection{Estimates for $\bar{\mathcal{K}}_0(t)$ and $\bar{\mathcal{K}}_1(t)$}\hfill
\label{subsubsec:integralestimates}

\textbf{For parameters $\lambda>0.$}
\vspace{0.02in}

For $\lambda>0$, we set $b=\frac{1}{\lambda}$ into Lemma \ref{lem:general}(1)-(3) to obtain 
\begin{equation}
\label{eq:firstint1}
\bar{\mathcal{K}}_0(t)\sim
\begin{cases}
C,\,\,\,\,\,\,\,\,\,\,\,\,\,&\,\lambda>q>\frac{1}{2}\,\,\,\,\,\,\text{or}\,\,\,\,\,q\in(0,1/2),\,\,\,\,\lambda>\frac{1}{2},
\\
C_3\mathcal{J}(\overline\alpha_i,t)^{\frac{1}{q}-\frac{1}{\lambda}},\,\,\,\,\,\,&q>0,\,\,\,\,\lambda\in(0,q)
\end{cases}
\end{equation}
for $\eta_*-\eta>0$ small and positive constants $C_3$ given by
\begin{equation}
\label{eq:cst1}
C_3=\frac{2m\Gamma\left(1+\frac{1}{q}\right)\Gamma\left(\frac{1}{\lambda}-\frac{1}{q}\right)}{\Gamma\left(\frac{1}{\lambda}\right)}\left(\frac{M_0}{\left|C_1\right|}\right)^{\frac{1}{q}}.
\end{equation}
Also, in (\ref{eq:firstint1})i) we assume that $\lambda$ and $q$ satisfy, whenever applicable, 
\begin{equation}
\label{eq:lemmaass1}
\lambda\neq\frac{q}{1-nq},\,\,\,\,\,\,\,\,q\neq\frac{1}{n}\,\,\,\,\,\,\,\forall\,\,\,\,\,\,\,n\in\mathbb{N}.
\end{equation}
We note that corresponding estimates for the missing values may be obtained via a simple continuity argument. 

Similarly, taking $b=1+\frac{1}{\lambda}$ we find 
\begin{equation}
\label{eq:secint1}
\bar{\mathcal{K}}_1(t)\sim
\begin{cases}
C,\,\,&q\in(1/2,1),\,\,\lambda>\frac{q}{1-q}\,\,\,\,\,\,\text{or}\,\,\,\,\,\,q\in(0,1/2),\,\,\lambda>1,
\\
C_4\mathcal{J}(\overline\alpha_i,t)^{\frac{1}{q}-\frac{1}{\lambda}-1},&q\in(0,1),\,\,0<\lambda<\frac{q}{1-q}\,\,\,\,\,\text{or}\,\,\,\,\,q\geq1,\,\,\lambda>0
\end{cases}
\end{equation}
with positive constants $C_4$ determined by
\begin{equation}
\label{eq:const2}
C_4=\frac{2m\Gamma\left(1+\frac{1}{q}\right)\Gamma\left(1+\frac{1}{\lambda}-\frac{1}{q}\right)}{\Gamma\left(1+\frac{1}{\lambda}\right)}\left(\frac{M_0}{\left|C_1\right|}\right)^{\frac{1}{q}}.
\end{equation}
Additionally, for (\ref{eq:secint1})i) we assume that $\lambda$ and $q$ satisfy (\ref{eq:lemmaass1}). 

\vspace{0.02in}
\textbf{For parameters $\lambda<0.$}
\vspace{0.02in}

For $\lambda<0$ and $b=\frac{1}{\lambda}$, Lemma \ref{lem:general}(3) implies that
\begin{equation}
\label{eq:thirdint1}
\bar{\mathcal{K}}_0(t)\sim C
\end{equation}
for $\eta_*-\eta>0$ small. Similarly, parts (2) and (3), now with $b=1+\frac{1}{\lambda}$, yield
\begin{equation}
\label{eq:lastest}
\bar{\mathcal{K}}_1(t)\sim C
\end{equation}
for either
\begin{equation}
\label{eq:condthirdint1}
\begin{cases}
q>0,\,\,\,\,&\lambda\in[-1,0),
\\
q\in(0,1),\,\,\,\,&\lambda<-1\,\,\,\,\,\,\,\,\,\,\,\,\text{satisfying}\,\,\,\,(\ref{eq:lemmaass1}),
\\
q>1,\,\,\,\,&\frac{q}{1-q}<\lambda<-1,
\end{cases}
\end{equation}
whereas
\begin{equation}
\label{eq:lastest1}
\bar{\mathcal{K}}_1(t)\sim C_5\mathcal{J}(\underline\alpha_j,t)^{\frac{1}{q}-\frac{1}{\lambda}-1}
\end{equation}
for $q>1$,\, $\lambda<\frac{q}{1-q}$ and positive constants $C_5$ determined by
\begin{equation}
\label{eq:const3}
C_5=\frac{2n\Gamma\left(1+\frac{1}{q}\right)\Gamma\left(1+\frac{1}{\lambda}-\frac{1}{q}\right)}{\Gamma\left(1+\frac{1}{\lambda}\right)}\left(\frac{\left|m_0\right|}{C_2}\right)^{\frac{1}{q}}.
\end{equation}

\subsubsection{$L^{\infty}$ Regularity for $\lambda\in\mathbb{R}^+\cup\{0\}, q\in\mathbb{R}^+$}\hfill
\label{subsubsec:generalcaselambdapospinf}

In this section, we use the estimates in \S\ref{subsubsec:integralestimates} to examine the $L^{\infty}$ regularity of $u_x$ for $\lambda\in\mathbb{R}^+\cup\{0\}$ and $u_0'$ satisfying (\ref{eq:expnew0}) for some $q\in\mathbb{R}^+$. Furthermore, the behaviour of the jacobian (\ref{eq:sum}) is also studied.

\begin{theorem}
\label{thm:lambdapos}
Consider the initial boundary value problem (\ref{eq:nonhomo})-(\ref{eq:pbc}) for $u_0'(\alpha)$ satisfying (\ref{eq:expnew0}). 
\begin{enumerate}

\item\label{it:global} If $q\in\mathbb{R}^+$ and $\lambda\in[0,q/2],$ solutions exist globally in time. More particularly, these vanish as $t\uparrow t_*=+\infty$ for $\lambda\in(0,q/2)$ but converge to a nontrivial steady state if $\lambda=q/2$.

\item\label{it:blow1} If $q\in\mathbb{R}^+$ and $\lambda\in(q/2,q)$, there exists a finite $t_*>0$ such that both the maximum $M(t)$ and the minimum $m(t)$ diverge to $+\infty$ and respectively to $-\infty$ as $t\uparrow t_*$. Moreover, $\lim_{t\uparrow t_*}u_x(\gamma(\alpha,t),t)=-\infty$ for $\alpha\notin\bigcup_{i,j}\{\overline\alpha_i\}\cup\{\underline\alpha_j\}$ (two-sided, everywhere blow-up).

\item\label{it:blow2} For $q\in(0,1/2)$ and $\lambda>1$ such that $q\neq\frac{1}{n}$ and $\lambda\neq\frac{q}{1-nq}$ for all $n\in\mathbb{N},$ there is a finite $t_*>0$ such that only the maximum blows up, $M(t)\to+\infty,$ as $t\uparrow t_*$ (one-sided, discrete blow-up). Further, if $\frac{1}{2}<\lambda<\frac{q}{1-q}$ for $q\in(1/3,1/2)$, a two-sided, everywhere blow-up (as described in (\ref{it:blow1}) above) occurs at a finite $t_*>0$.

\item\label{it:blow3} Suppose $q\in(1/2,1)$. Then for $q<\lambda<\frac{q}{1-q}$, there exists a finite $t_*>0$ such that, as $t\uparrow t_*$, two-sided, everywhere blow-up develops. If instead $\lambda>\frac{q}{1-q}$, only the maximum diverges, $M(t)\to+\infty$, as $t\uparrow t_*<+\infty$. 

\item\label{it:blow4} For $\lambda>q>1$, there is a finite $t_*>0$ such that $u_x$ undergoes a two-sided, everywhere blow-up as $t\uparrow t_*$. 
\end{enumerate}
\begin{proof}
Suppose $\lambda,q>0$, let $C$ denote a positive constant which may depend on $\lambda$ and $q$, and set $\eta_*=\frac{1}{\lambda M_0}.$

\vspace{0.02in}
\textbf{Proof of Statements} (\ref{it:global}) \textbf{and} (\ref{it:blow1})
\vspace{0.02in}

Suppose $\lambda\in(0,q)$ for some $q>0$. Then, for $\eta_*-\eta>0$ small $\bar{\mathcal{K}}_0(t)$ satisfies (\ref{eq:firstint1})ii) while $\bar{\mathcal{K}}_1(t)$ obeys (\ref{eq:secint1})ii). Consequently, (\ref{eq:mainsolu}) implies that 
\begin{equation}
\label{eq:ok3} 
u_x(\gamma(\alpha,t),t)\sim \frac{M_0}{C_3^{^{2\lambda}}}\left(\frac{\mathcal{J}(\overline\alpha_i,t)}{\mathcal{J}(\alpha,t)}-\frac{C_4}{C_3}\right)\mathcal{J}(\overline\alpha_i,t)^{1-\frac{2\lambda}{q}}
\end{equation}
for positive constants $C_3$ and $C_4$ given by (\ref{eq:cst1}) and (\ref{eq:const2}). But for $y_1=\frac{1}{\lambda}-\frac{1}{q}$ and $y_2=\frac{1}{\lambda}$, (\ref{eq:gammarel})ii), (\ref{eq:cst1}) and (\ref{eq:const2}) yield
 \begin{equation}
\label{eq:gammaid}
\frac{C_4}{C_3}=\frac{\Gamma(y_1+1)\,\Gamma(y_2)}{\Gamma(y_1)\,\Gamma(y_2+1)}=\frac{y_1}{y_2}=1-\frac{\lambda}{q}\in(0,1),\,\,\,\,\,\,\,\,\,\,\,\,\,\,\,\lambda\in(0,q).
\end{equation}
As a result, setting $\alpha=\overline\alpha_i$ in (\ref{eq:ok3}) and using (\ref{eq:maxmin})i) implies that
\begin{equation}
\label{eq:ux5}
M(t)\sim\frac{M_0}{C_3^{^{2\lambda}}}\left(\frac{\lambda}{q}\right)\mathcal{J}(\overline\alpha_i,t)^{1-\frac{2\lambda}{q}}
\end{equation}
for $\eta_*-\eta>0$ small, whereas, if $\alpha\neq\overline\alpha_i$,
\begin{equation}
\label{eq:ux6}
u_x(\gamma(\alpha,t),t)\sim-\left(1-\frac{\lambda}{q}\right)\frac{M_0}{C_3^{^{2\lambda}}}\mathcal{J}(\overline\alpha_i,t)^{1-\frac{2\lambda}{q}}.
\end{equation}
Clearly, when $\lambda=q/2$, 
$$M(t)\to\frac{M_0}{2C_3^{\,q}}>0$$
as $\eta\uparrow\eta_*$, while, for $\alpha\neq\overline\alpha_i$,
$$u_x(\gamma(\alpha,t),t)\to-\frac{M_0}{2C_3^{\,q}}<0.$$
If $\lambda\in(0,q/2)$, (\ref{eq:ux5}) now implies that
$$M(t)\to0^+$$
as $\eta\uparrow\eta_*$, whereas, using (\ref{eq:ux6}) for $\alpha\neq\overline\alpha_i$, 
$$u_x(\gamma(\alpha,t),t)\to0^-.$$
In contrast, if $\lambda\in(q/2,q)$, $1-\frac{2\lambda}{q}<0$. Then (\ref{eq:ux5}) and (\ref{eq:ux6}) yield
\begin{equation}
\label{eq:ux52}
M(t)\to+\infty
\end{equation}
as $\eta\uparrow\eta_*$, but 
\begin{equation}
\label{eq:ux62}
u_x(\gamma(\alpha,t),t)\to-\infty
\end{equation}
for $\alpha\neq\overline\alpha_i$. Lastly, rewriting (\ref{eq:etaivp}) as
\begin{equation}
\label{eq:time}
dt=\bar{\mathcal{K}}_0(t)^{2\lambda}d\eta
\end{equation}
and using (\ref{eq:firstint1})ii), we obtain
\begin{equation}
\label{eq:ux7}
t_*-t\sim C\int_{\eta}^{\eta_*}{(1-\lambda\mu M_0)^{\frac{2\lambda}{q}-2}d\mu}
\end{equation}
or equivalently
\begin{equation}
\label{eq:timegen}
t_*-t\sim 
\begin{cases}
\frac{C}{2\lambda-q}\left(C(\eta_*-\eta)^{\frac{2\lambda}{q}-1}-\lim_{\mu\uparrow\eta_*}(\eta_*-\mu)^{\frac{2\lambda}{q}-1}\right),\,\,\,\,&\lambda\in(0,q)\backslash\{q/2\},
\\
C\left(\log(\eta_*-\eta)-\lim_{\mu\uparrow\eta_*}\log(\eta_*-\mu)\right),\,\,\,&\lambda=q/2.
\end{cases}
\end{equation}
Consequently, $t_*=+\infty$ for $\lambda\in(0,q/2]$, while $0<t_*<+\infty$ if $\lambda\in(q/2,q)$. Lastly, the case $\lambda=0$ follows from the results in \cite{Sarria1}. 

\vspace{0.02in}
\textbf{Proof of Statement} (\ref{it:blow2})
\vspace{0.02in}

First, suppose $q\in(0,1/2)$ and $\lambda>1$ satisfy (\ref{eq:lemmaass1}). Then $\bar{\mathcal{K}}_0(t)$ and $\bar{\mathcal{K}}_1(t)$ satisfy (\ref{eq:firstint1})i)  and (\ref{eq:secint1})i), respectively. Therefore, (\ref{eq:mainsolu}) implies that
\begin{equation}
\label{eq:blow3case}
u_x(\gamma(\alpha,t),t)\sim C\left(\frac{1}{\mathcal{J}(\alpha,t)}-C\right)
\end{equation}
for $\eta_*-\eta>0$ small. Setting $\alpha=\overline\alpha_i$ into (\ref{eq:blow3case}) and using (\ref{eq:maxmin})i) gives
$$M(t)\sim\frac{C}{\mathcal{J}(\overline\alpha_i,t)}\to+\infty$$
as $\eta\uparrow\eta_*$, while, if $\alpha\neq\overline\alpha_i$, $u_x(\gamma(\alpha,t),t)$ remains finite for all $\eta\in[0,\eta_*]$ due to the definition of $M_0$. The existence of a finite blow-up time $t_*>0$ for the maximum is guaranteed by (\ref{eq:firstint1})i) and (\ref{eq:time}), which lead to
\begin{equation}
\label{eq:timeblow3case}
t_*-t\sim C(\eta_*-\eta).
\end{equation}
Next, suppose $\frac{1}{2}<\lambda<\frac{q}{1-q}$ for $q\in(1/3,1/2)$, so that $\frac{q}{1-q}\in(1/2,1)$. Then, using (\ref{eq:firstint1})i) and (\ref{eq:secint1})ii) in (\ref{eq:mainsolu}), we find that
\begin{equation}
\label{eq:blow4case}
u_x(\gamma(\alpha,t),t)\sim C\left(\frac{C}{\mathcal{J}(\alpha,t)}-\mathcal{J}(\overline\alpha_i,t)^{^{\frac{1}{q}-\frac{1}{\lambda}-1}}\right)
\end{equation}
for $\eta_*-\eta>0$ small. Set $\alpha=\overline\alpha_i$ into the above and use $\lambda>q$ to obtain
\begin{equation}
\label{eq:blow4case1}
M(t)\sim\frac{C}{\mathcal{J}(\overline\alpha_i,t)}\to+\infty
\end{equation}
as $\eta\uparrow\eta_*$. On the other hand, for $\alpha\neq\overline\alpha_i$, the space-dependent in (\ref{eq:blow4case}) now remains finite for all $\eta\in[0,\eta_*]$. As a result, the second term dominates and
\begin{equation}
\label{eq:blow4case2}
u_x(\gamma(\alpha,t),t)\sim-C\mathcal{J}(\overline\alpha_i,t)^{^{\frac{1}{q}-\frac{1}{\lambda}-1}}\to-\infty
\end{equation}
as $\eta\uparrow\eta_*$. The existence of a finite blow-up time $t_*>0$, follows, as in the previous case, from (\ref{eq:time}) and (\ref{eq:firstint1})i). 

\vspace{0.02in}
\textbf{Proof of Statement} (\ref{it:blow3})
\vspace{0.02in}

Part (\ref{it:blow3}) follows from an argument analogous to the one above. Briefly, if $q<\lambda<\frac{q}{1-q}$ for $q\in(1/2,1)$, we use estimates (\ref{eq:firstint1})i) and (\ref{eq:secint1})ii) on (\ref{eq:mainsolu}) to get (\ref{eq:blow4case}), with different positive constants $C$. Two-sided, everywhere blow-up in finite-time then follows just as above. If instead $\lambda>\frac{q}{1-q}$ for $q\in(1/2,1)$, then (\ref{eq:firstint1})i) still holds but $\bar{\mathcal{K}}_1(t)$ now remains bounded for all $\eta\in[0,\eta_*]$; it satisfies (\ref{eq:secint1})i). Therefore, up to different positive constants $C$, (\ref{eq:mainsolu}) leads to (\ref{eq:blow3case}), and so only the maximum diverges, $M(t)\to+\infty$, as $t$ approaches some finite $t_*>0$ whose existence is guaranteed by (\ref{eq:timeblow3case}).

\vspace{0.02in}
\textbf{Proof of Statement} (\ref{it:blow4})
\vspace{0.02in}

For $\lambda>q>1$, (\ref{eq:firstint1})i), (\ref{eq:secint1})ii) and (\ref{eq:mainsolu}) imply (\ref{eq:blow4case}). Then, we follow the argument used to establish the second part of (\ref{it:blow2}) to show that two-sided, everywhere finite-time blow-up occurs. See \S \ref{sec:examples} for examples.\end{proof}
\end{theorem}

\begin{remark}
Theorems \ref{thm:p=1} and \ref{thm:lambdapos} allow us to predict the regularity of stagnation point-form (SPF) solutions to the two ($\lambda=1$) and three ($\lambda=1/2$) dimensional incompressible Euler equations assuming we know something about the curvature of the initial data $u_0$ near $\overline\alpha_i$. Setting $\lambda=1$ into Theorem \ref{thm:lambdapos}(\ref{it:global}) implies that SPF solutions in the 2D setting persist for all time if $u_0'$ satisfies (\ref{eq:expnew0}) for arbitrary $q\geq2$. On the contrary, Theorems \ref{thm:p=1} and \ref{thm:lambdapos}(\ref{it:blow1})-(\ref{it:blow3}), tell us that if $q\in(1/2,2)$, two-sided, everywhere finite-time blow-up occurs. Analogously, solutions to the corresponding 3D problem exist globally in time for $q\geq1$, whereas, two-sided, everywhere blow-up develops when $q\in(1/2,1)$. See Table \ref{table:thecase} below. Finally, we remark that finite-time blow-up in $u_x$ is expected for both the two and three dimensional equations if $q\in(0,1/2]$. See for instance \S\ref{sec:examples} for a blow-up example in the 3D case with $q=1/3$.
\end{remark}

\begin{table}[ht]
\caption{Regularity of SPF solutions to Euler equations}
\centering
		\renewcommand{\arraystretch}{1.2}
    \begin{tabular}{ | m{2.0cm} | m{3.4cm} | m{3.4cm} | }
    \hline
    \centering{$q$} & \centering{2D Euler}  & $\text{\,\,\,\,\,\,\,\,\,\,\,\,\,\,3D Euler}$ \\ \hline
    \centering{$(1/2,1)$} &$\text{Finite time blow up}$ & $\text{Finite time blow up}$  \\ \hline
    \centering{$[1,2)$} &$\text{Finite time blow up}$ & $\text{\,\,\,\,\,\,Global in time}$     \\ \hline
    \centering{$[2,+\infty)$} &$\text{\,\,\,\,\,\,\,\,Global in time}$ &$\text{\,\,\,\,\,\,Global in time}$ \\ \hline
    \end{tabular}
\label{table:thecase}
\end{table}
Corollary \ref{coro:lambda>1/2} below briefly examines the behaviour, as $t\uparrow t_*$, of the jacobian (\ref{eq:sum}) for $t_*>0$ as in Theorem \ref{thm:lambdapos}.

\begin{corollary}
\label{coro:lambda>1/2}
Consider the initial boundary value problem (\ref{eq:nonhomo})-(\ref{eq:pbc}) with $u_0'(\alpha)$ satisfying (\ref{eq:expnew0}) for $q\in\mathbb{R}^+$, and let $t_*>0$ be as in Theorem \ref{thm:lambdapos}. It follows,

\begin{enumerate}

\item\label{it:item1} For $q>0$ and $\lambda\in(0,q)$,
\begin{equation}
\label{eq:jacone}
\lim_{t\uparrow t_*}\gamma_{\alpha}(\alpha,t)=
\begin{cases}
+\infty,\,\,\,\,\,\,\,\,\,\,\,\,\,\,\,&\alpha=\overline\alpha_i,
\\
0,\,\,\,\,\,\,\,\,&\alpha\neq\overline\alpha_i
\end{cases}
\end{equation}
where $t_*=+\infty$ for $\lambda\in(0,q/2]$, while $0<t_*<+\infty$ if $\lambda\in(q/2,q)$.

\item\label{it:item1.5} Suppose $\lambda>q>1/2$, or $q\in(0,1/2)$ and $\lambda>1/2$, satisfy (\ref{eq:lemmaass1}). Then, there exists a finite $t_*>0$ such that
\begin{equation}
\label{eq:jacone22}
\lim_{t\uparrow t_*}\gamma_{\alpha}(\alpha,t)=
\begin{cases}
+\infty,\,\,\,\,\,\,\,\,\,\,&\alpha=\overline\alpha_i,
\\
C(\alpha),\,\,\,\,\,\,\,\,&\alpha\neq\overline\alpha_i
\end{cases}
\end{equation}
where $C(\alpha)\in\mathbb{R}^+$ depends on the choice of $\lambda$, $q$ and $\alpha\neq\overline\alpha_i$.
\end{enumerate}
\end{corollary}
\begin{proof}
The limits (\ref{eq:jacone}) and (\ref{eq:jacone22}) follow straightforwardly from (\ref{eq:sum}) and estimates (\ref{eq:firstint1})ii) and (\ref{eq:firstint1})i), respectively; whereas, the finite or infinite character of $t_*>0$ is a consequence of Theorem \ref{thm:lambdapos}.\end{proof}






\subsubsection{Further $L^p$ Regularity for $\lambda\in[0,+\infty), q\in\mathbb{R}^+$ and $p\in[1,+\infty)$}\hfill
\label{subsubsec:generalcaselambdapospfin}

From Theorem \ref{thm:lambdapos}, if $\lambda\in[0,q/2]$ for $q\in\mathbb{R}^+$, solutions remain in $L^{\infty}$ for all time; otherwise, $\left\|u_x\right\|_{\infty}$ diverges as $t$ approaches some finite $t_*>0$. In this section, we study further properties of $L^p$ regularity in $u_x$, as $t\uparrow t_*$, for $\lambda>q/2$, $p\in[1,+\infty)$ and initial data $u_0'(\alpha)$ satisfying (\ref{eq:expnew0}). To do so, we will use the upper and lower bounds (\ref{eq:upper}) and (\ref{eq:lower}). Consequently, for $\eta_*-\eta>0$ small and $\eta_*=\frac{1}{\lambda M_0}$, estimates on the behaviour of the time-dependent integrals
\begin{equation}
\label{eq:remaining}
\int_0^1{\frac{d\alpha}{\mathcal{J}(\alpha,t)^{\frac{1}{\lambda p}}}},\,\,\,\,\,\,\,\,\,\,\,\,\,\int_0^1{\frac{d\alpha}{\mathcal{J}(\alpha,t)^{1+\frac{1}{\lambda p}}}},\,\,\,\,\,\,\,\,\,\,\,\,\,\,\,\int_0^1{\frac{d\alpha}{\mathcal{J}(\alpha,t)^{p+\frac{1}{\lambda}}}}
\end{equation}
are required. Since these may be obtained directly from Lemma \ref{lem:general}(1)-(3), we omit the details and state our findings below. For $p\in[1,+\infty)$,
\begin{equation}
\label{eq:lpos3}
\int_0^1{\frac{d\alpha}{\mathcal{J}(\alpha,t)^{\frac{1}{\lambda p}}}}\sim
\begin{cases}
C,\,\,&q\in(0,1/2),\, \lambda>\frac{1}{2p}\,\,\,\,\text{or}\,\,\,q>\frac{1}{2},\,\,\, \lambda>\frac{q}{p},
\\
C_6\mathcal{J}(\overline\alpha,t)^{\frac{1}{q}-\frac{1}{\lambda p}}, &q>0,\,\,\, \lambda\in(0,q/p)
\end{cases} 
\end{equation}
with positive constants
\begin{equation}
\label{eq:c6}
C_6=\frac{2\Gamma\left(1+\frac{1}{q}\right)\Gamma\left(\frac{1}{\lambda p}-\frac{1}{q}\right)}{\Gamma\left(\frac{1}{\lambda p}\right)}\left(\frac{M_0}{\left|C_1\right|}\right)^{\frac{1}{q}}.
\end{equation}
Also
\begin{equation}
\label{eq:lpos4}
\int_0^1{\frac{d\alpha}{\mathcal{J}(\alpha,t)^{1+\frac{1}{\lambda p}}}}\sim C
\end{equation}
for either 
\begin{equation}
\label{eq:lpos5}
\begin{cases}
&q\in(0,1/2),\,\,\,\,\,\, \lambda>\frac{1}{p},
\\
&q\in(1/2,1),\,\,\,\,\,\, \lambda>\frac{q}{p(1-q)},
\end{cases}
\end{equation}
whereas
\begin{equation}
\label{eq:lpos6}
\int_0^1{\frac{d\alpha}{\mathcal{J}(\alpha,t)^{1+\frac{1}{\lambda p}}}}\sim C_7\mathcal{J}(\overline\alpha,t)^{\frac{1}{q}-\frac{1}{\lambda p}-1}
\end{equation}
for   
\begin{equation}
\label{eq:lpos7}
\begin{cases}
&q\in(0,1),\,\,\,\,\, 0<\lambda<\frac{q}{p(1-q)},
\\
&q\geq1,\,\,\,\,\, \lambda>0.
\end{cases}
\end{equation}
The positive constants $C_7$ in (\ref{eq:lpos6}) are obtained by replacing every $\frac{1}{\lambda p}$ term in (\ref{eq:c6}) by $1+\frac{1}{\lambda p}$. Also, due to Lemma \ref{lem:general}, (\ref{eq:lpos3})i) and (\ref{eq:lpos4}) are valid for
\begin{equation}
\label{eq:req1lambdapos1}
\lambda\neq\frac{q}{p(1-nq)},\,\,\,\,\,\,\,\,\,q\neq\frac{1}{n}\,\,\,\forall\,\,\,\,n\in\mathbb{N}\cup\{0\},
\end{equation}
where a simple continuity argument may again be used (see (\ref{eq:lemmaass1})) to obtain estimates for the missing values. Finally 
\begin{equation}
\label{eq:lpos8}
\int_0^1{\frac{d\alpha}{\mathcal{J}(\alpha,t)^{p+\frac{1}{\lambda}}}}\sim C
\end{equation}
for either
\begin{equation}
\label{eq:lpos9}
\begin{cases}
&q\in(0,1/2),\,\,\,\,p\in[1,2),\,\,\,\,\,\,\,\,\,\, \lambda>\frac{1}{2-p},
\\
&q\in(1/2,1),\,\,\,\, p\in[1,1/q),\,\,\,\, \lambda>\frac{q}{1-pq},
\end{cases}
\end{equation}
while
\begin{equation}
\label{eq:lpos10}
\int_0^1{\frac{d\alpha}{\mathcal{J}(\alpha,t)^{p+\frac{1}{\lambda}}}}\sim C\mathcal{J}(\overline\alpha,t)^{\frac{1}{q}-\frac{1}{\lambda}-p}
\end{equation}
if  
\begin{equation}
\label{eq:lpos11}
\begin{cases}
&q\in(0,1],\,\,\,\,\,\, p\in[1,1/q),\,\,\,\, 0<\lambda<\frac{q}{1-pq},
\\
&q\in(0,1],\,\,\,\,\,\, p\geq\frac{1}{q},\,\,\,\,\,\, \lambda>0,
\\
&q>1,\,\,\,\,\,\,\,\,\,\,\,\,\,\, p\geq1,\,\,\,\,\,\,\, \lambda>0.
\end{cases}
\end{equation}
Estimate (\ref{eq:lpos8}) is in turn valid for
\begin{equation}
\label{eq:req1lambdapos2}
\lambda\neq\frac{q}{1+q(n-p)},\,\,\,\,\,\,\,\,\,q\neq\frac{1}{n}\,\,\,\,\forall\,\,\,n\in\mathbb{N}.
\end{equation}
In what follows, $t_*>0$ will denote the $L^{\infty}$ blow-up time for $u_x$ in Theorem \ref{thm:lambdapos}. Also, we will assume that (\ref{eq:lemmaass1}), (\ref{eq:req1lambdapos1}) and (\ref{eq:req1lambdapos2}) hold whenever their corresponding estimates are used. 
We begin by considering the lower bound (\ref{eq:lower}). In particular, we will show that two-sided, everywhere blow-up in Theorem \ref{thm:lambdapos} corresponds to a diverging $\left\|u_x\right\|_p$ for all $p>1$. Then, by studying the upper bound (\ref{eq:upper}), we will find that if $q\in\mathbb{R}^+$ and $\lambda>q$ are such that only the maximum diverges at a finite $t_*>0$, then $u_x$ remains integrable for all $t\in[0,t_*]$, whereas, its regularity in smaller $L^p$ spaces for $t\in[0,t_*]$ will vary according to the value of the parameter $\lambda$ as a function of either $p$, $q$, or both.

Suppose $q/2<\lambda<q/p$ for $q\in\mathbb{R}^+$ and $p\in(1,2)$. Then (\ref{eq:lpos3})ii) holds as well as (\ref{eq:firstint1})ii), since $(q/2,q/p)\subset(0,q)$. Now, if $q\in(0,1)$ then $0<\frac{q}{2}<\lambda<\frac{q}{p}<q<\frac{q}{1-q},$ and so (\ref{eq:secint1})ii) applies, otherwise, (\ref{eq:secint1})ii) also holds for $q\geq1$ and $\lambda>0$. Similarly for $q\in(0,1)$, we have that $0<\frac{q}{2}<\lambda<\frac{q}{p}<\frac{q}{p(1-q)}$ so that (\ref{eq:lpos6}) is valid. Alternatively, this last estimate also holds if $q\geq1$ for $\lambda>0$. Accordingly, using these estimates in (\ref{eq:lower}) yields, after simplification,
\begin{equation*}
\begin{split}
\left\|u_x(\cdot,t)\right\|_p&\geq\frac{1}{\left|\lambda\eta(t)\right|\bar{\mathcal{K}}_0(t)^{^{2\lambda+\frac{1}{p}}}}\left|\int_0^1{\frac{d\alpha}{\mathcal{J}(\alpha,t)^{^{1+\frac{1}{\lambda p }}}}}-\frac{\bar{\mathcal{K}}_1(t)}{\bar{\mathcal{K}}_0(t)}\int_0^1{\frac{d\alpha}{\mathcal{J}(\alpha,t)^{\frac{1}{\lambda p}}}}\right|
\\
&\sim C(p-1)\mathcal{J}(\overline\alpha,t)^{\sigma(p,q,\lambda)}
\end{split}
\end{equation*} 
for $\eta_*-\eta>0$ small and $\sigma(p,q,\lambda)=1+\frac{1}{q}\left(1-\frac{1}{p}-2\lambda\right).$ Consequently, $\left\|u_x\right\|_p$ will diverge as $\eta\uparrow\eta_*$ if $\sigma(p,q,\lambda)<0$, or equivalently for $p(1+q-2\lambda)-1<0$. Since $q/2<\lambda<q/p$ for $q>0$ and $p\in(1,2)$, we find this to be the case as long as 
$$q\in\mathbb{R}^+,\,\,\,\,\,\,\,\,\,1<p<1+\frac{q}{1+q},\,\,\,\,\,\,\,\,\,\,\,\frac{1}{2}\left(q+1-\frac{1}{p}\right)<\lambda<\frac{q}{p}.$$
Therefore, by taking $p-1>0$ arbitrarily small, we find that
$$\lim_{t\uparrow t_*}\left\|u_x(\cdot,t)\right\|_p=+\infty$$
for $\lambda\in(q/2,q)$ and $q>0$. The existence of a finite blow-up time $t_*>0$ follows from Theorem \ref{thm:lambdapos}(\ref{it:blow1}), while the embedding 
\begin{equation}
\label{eq:embedding}
\begin{split}
L^s\hookrightarrow L^p,\,\,\,\,\,\,\,\,\,\,\,\,\,\,\,s\geq p,
\end{split}
\end{equation}
yields $L^p$ blow-up for any $p>1$. Next, for $q\in(1/3,1/2)$ we consider values of $\lambda$ lying between stagnation point-form solutions to the 2D ($\lambda=1$) and 3D ($\lambda=1/2$) incompressible Euler equations. Suppose $\frac{1}{2}<\lambda<\frac{q}{p(1-q)}$ for $1<p<\frac{2q}{1-q}$ and $q\in(1/3,1/2)$. The condition on $p$ simply guarantees that $\frac{q}{p(1-q)}>\frac{1}{2}$ for $q$ as specified. Furthermore, we have that
$$0<\frac{1}{2p}<\frac{1}{2}<\lambda<\frac{q}{p(1-q)}<\frac{q}{1-q}\in(1/2,1),$$
so that relative to our choice of $\lambda$ and $q$, $\lambda\in(1/2,1).$ Using the above, we find that (\ref{eq:firstint1})i), (\ref{eq:secint1})ii), (\ref{eq:lpos3})i) and (\ref{eq:lpos6}) hold, and so (\ref{eq:lower}) leads to 
\begin{equation}
\label{eq:touse1}
\begin{split}
\left\|u_x(\cdot,t)\right\|_p&\geq\frac{1}{\left|\lambda\eta(t)\right|\bar{\mathcal{K}}_0(t)^{^{2\lambda+\frac{1}{p}}}}\left|\int_0^1{\frac{d\alpha}{\mathcal{J}(\alpha,t)^{^{1+\frac{1}{\lambda p }}}}}-\frac{\bar{\mathcal{K}}_1(t)}{\bar{\mathcal{K}}_0(t)}\int_0^1{\frac{d\alpha}{\mathcal{J}(\alpha,t)^{\frac{1}{\lambda p}}}}\right|
\\
&\sim C\left|C\mathcal{J}(\overline\alpha,t)^{\frac{1}{q}-\frac{1}{\lambda p}-1}-\mathcal{J}(\overline\alpha,t)^{\frac{1}{q}-\frac{1}{\lambda}-1}\right|
\\
&\sim C\mathcal{J}(\overline\alpha,t)^{\frac{1}{q}-\frac{1}{\lambda}-1}
\end{split}
\end{equation} 
for $\eta_*-\eta>0$ small. Therefore, as $\eta\uparrow\eta_*$, $\left\|u_x\right\|_p$ will diverge for all $\frac{1}{2}<\lambda<\frac{q}{p(1-q)}$, $q\in(1/3,1/2)$ and $1<p<\frac{2q}{1-q}$. Here, we can take $p-1>0$ arbitrarily small and use (\ref{eq:embedding}) to conclude the finite-time blow-up, as $t\uparrow t_*$, of $\left\|u_x\right\|_p$ for all $\frac{1}{2}<\lambda<\frac{q}{1-q}$,  $q\in(1/3,1/2)$ and $p>1$. The existence of a finite blow-up time $t_*>0$ is guaranteed by the second part of Theorem \ref{thm:lambdapos}(\ref{it:blow2}). Now suppose $q\in(1/2,1)$ and $q<\lambda<\frac{q}{p(1-q)}$ for $1<p<\frac{1}{1-q}$. This means that $\lambda>q>1/2$ and
\begin{equation}
0<\frac{q}{p}<q<\lambda<\frac{q}{p(1-q)}<\frac{q}{1-q}.
\end{equation}
Consequently, using (\ref{eq:firstint1})i), (\ref{eq:secint1})ii), (\ref{eq:lpos3})i) and (\ref{eq:lpos6}) in (\ref{eq:lower}), implies (\ref{eq:touse1}), possibly with distinct positive constants $C$. Then, as $\eta\uparrow\eta_*$,
$$\left\|u_x\right\|_p\to+\infty$$
for all $q<\lambda<\frac{q}{p(1-q)}$, $q\in(1/2,1)$ and $1<p<\frac{1}{1-q}$. Similarly, if $q$ and $p$ are as above, but $\frac{q}{p(1-q)}<\lambda<\frac{q}{1-q}$, (\ref{eq:firstint1})i), (\ref{eq:secint1})ii), (\ref{eq:lpos3})i) and (\ref{eq:lpos4}) imply
\begin{equation*}
\begin{split}
\left\|u_x(\cdot,t)\right\|_p&\geq\frac{1}{\left|\lambda\eta(t)\right|\bar{\mathcal{K}}_0(t)^{^{2\lambda+\frac{1}{p}}}}\left|\int_0^1{\frac{d\alpha}{\mathcal{J}(\alpha,t)^{^{1+\frac{1}{\lambda p }}}}}-\frac{\bar{\mathcal{K}}_1(t)}{\bar{\mathcal{K}}_0(t)}\int_0^1{\frac{d\alpha}{\mathcal{J}(\alpha,t)^{\frac{1}{\lambda p}}}}\right|
\\
&\sim C\left|C-\mathcal{J}(\overline\alpha,t)^{\frac{1}{q}-\frac{1}{\lambda}-1}\right|
\\
&\sim C\mathcal{J}(\overline\alpha,t)^{\frac{1}{q}-\frac{1}{\lambda}-1}\to+\infty
\end{split}
\end{equation*}  
as $\eta\uparrow\eta_*$. From these last two results and (\ref{eq:embedding}), we see that, as $\eta\uparrow\eta_*$, $\left\|u_x\right\|_p\to+\infty$ for all $q<\lambda<\frac{q}{1-q}$, $q\in(1/2,1)$ and $p>1$. The existence of a finite $t_*>0$ follows from Theorem \ref{thm:lambdapos}(\ref{it:blow3}). Lastly, suppose $\lambda>q>1$ and $p>1$. Then, estimates (\ref{eq:firstint1})i), (\ref{eq:secint1})ii), (\ref{eq:lpos3})i) and (\ref{eq:lpos6}) hold for $\eta_*-\eta>0$ small. As a result, (\ref{eq:lower}) implies (\ref{eq:touse1}), which in turn leads to $L^p$ blow-up of $u_x$ for any $\lambda>q>1$ and $p>1$, as $\eta\uparrow\eta_*$. The existence of a finite $t_*>0$ is due to Theorem \ref{thm:lambdapos}(\ref{it:blow4}).

Notice from the results established so far, that some values of $\lambda>q/2$ for $q>0$ are missing. These are precisely the cases for which the lower bound (\ref{eq:lower}) yields inconclusive information about the $L^p$ regularity of $u_x$ for $p\in(1,+\infty)$. To examine some aspects of the $L^p$ regularity of $u_x$ for $t\in[0,t_*]$ and $p\in[1,+\infty)$ in these particular cases, we consider the upper bound (\ref{eq:upper}). First, suppose $q\in(0,1/2)$ and $\lambda>\frac{1}{2-p}$ for $p\in[1,2)$. Then $\lambda>\frac{1}{2-p}>1>\frac{q}{1-q}>q,$ so that (\ref{eq:firstint1})i), (\ref{eq:secint1})i) and (\ref{eq:lpos8}), imply that the integral terms in (\ref{eq:upper}) remain bounded, and nonzero, for $\eta\in[0,\eta_*]$. We conclude that 
\begin{equation}
\label{eq:upp1}
\begin{split}
\lim_{t\uparrow t_*}\left\|u_x(\cdot,t)\right\|_p<+\infty
\end{split}
\end{equation}
for all $\lambda>\frac{1}{2-p}$, $q\in(0,1/2)$ and $p\in[1,2)$. Here, $t_*>0$ denotes the finite $L^{\infty}$ blow-up time for $u_x$ established in the first part of Theorem \ref{thm:lambdapos}(\ref{it:blow2}). Particularly, this result implies that even though $\lim_{t\uparrow t_*}\left\|u_x\right\|_{\infty}=+\infty$ for all $\lambda>1$ when $q\in(0,1/2)$, $u_x$ remains integrable for $t\in[0,t_*]$. Finally, suppose $q\in(1/2,1)$ and $\lambda>\frac{q}{1-pq}$ for $p\in[1,1/q)$. Then $\lambda>\frac{q}{1-pq}\geq\frac{q}{1-q}>1>q>\frac{1}{2}$,
and so (\ref{eq:firstint1})i), (\ref{eq:secint1})i) and (\ref{eq:lpos8}) hold. Consequently, (\ref{eq:upper}) implies that $\lim_{t\uparrow t_*}\left\|u_x\right\|_p<+\infty$ for all $\lambda>\frac{q}{1-pq}$, $q\in(1/2,1)$ and $p\in[1,1/q)$. This time, $t_*>0$ stands as the finite $L^{\infty}$ blow-up time for $u_x$ established in the second part of Theorem \ref{thm:lambdapos}(\ref{it:blow3}). Furthermore, this result tells us that even though $\lim_{t\uparrow t_*}\left\|u_x\right\|_{\infty}=+\infty$ for $\lambda>\frac{q}{1-q}$ and $q\in(1/2,1)$, $u_x$ stays integrable for all $t\in[0,t_*]$. These last two results on the integrability of $u_x$, for $t\in[0,t_*]$, become more apparent if we set $p=1$ in (\ref{eq:upper}) to obtain $$\left\|u_x(\cdot,t)\right\|_1\leq\frac{2\bar{\mathcal{K}}_1(t)}{\left|\lambda\eta(t)\right|\bar{\mathcal{K}}_0(t)^{1+2\lambda}}.$$ 
The result then follows from the above inequality and estimates (\ref{eq:firstint1})i) and (\ref{eq:secint1})i). Theorem \ref{thm:lplambdapos} below summarizes the above results. 

\begin{theorem}
\label{thm:lplambdapos}
Consider the initial boundary value problem (\ref{eq:nonhomo})-(\ref{eq:pbc}) for $u_0'(\alpha)$ satisfying (\ref{eq:expnew0}), and let $t_*>0$ be as in Theorem \ref{thm:lambdapos}.
\begin{enumerate}

\item\label{it:lplambdapos1} For $q>0$ and $\lambda\in[0,q/2]$, $\lim_{t\to+\infty}\left\|u_x\right\|_p<+\infty$ for all $p\geq1$. More particularly, $\lim_{t\to+\infty}\left\|u_x\right\|_p=0$ for $\lambda\in(0,q/2)$, while, as $t\to+\infty$, $u_x$ converges to a nontrivial, $L^{\infty}$ function when $\lambda=q/2$.

\item\label{it:lplambdapos2} Let $p>1$. Then, there exists a finite $t_*>0$ such that for all $q>0$ and $\lambda\in(q/2,q)$, $\lim_{t\uparrow t_*}\left\|u_x\right\|_p=+\infty$. Similarly for $\lambda>q>1$, or $\frac{1}{2}<\lambda<\frac{q}{1-q}$, $q\in(1/3,1/2)$.

\item\label{it:lplambdapos3} For all $q\in(0,1/2)$, $\lambda>\frac{1}{2-p}$ and $p\in[1,2)$, there exists a finite $t_*>0$ such that $\lim_{t\uparrow t_*}\left\|u_x\right\|_p<+\infty$ (see Theorem \ref{thm:lambdapos}(\ref{it:blow2})). 

\item\label{it:lplambdapos4} Suppose $q\in(1/2,1)$. Then, there exists a finite $t_*>0$ such that $\lim_{t\uparrow t_*}\left\|u_x\right\|_p=+\infty$ for $q<\lambda<\frac{q}{1-q}$ and $p>1$, whereas, if $\lambda>\frac{q}{1-pq}$ and $p\in[1,1/q)$, $\lim_{t\uparrow t_*}\left\|u_x\right\|_p<+\infty$ (see Theorem \ref{thm:lambdapos}(\ref{it:blow3})).
\end{enumerate}
\end{theorem}


\subsubsection{$L^{\infty}$ regularity for $\lambda<0$ and $q\in\mathbb{R}^+$}\hfill
\label{subsubsec:generalcaselambdanegpinf}

We now examine the $L^{\infty}$ regularity of $u_x$ for parameters $\lambda<0$ and initial data satisfying (\ref{eq:expnew00}) for arbitrary $q\in\mathbb{R}^+$. We prove Theorem \ref{thm:lambdaneg} below.
\begin{theorem}
\label{thm:lambdaneg}
Consider the initial boundary value problem (\ref{eq:nonhomo})-(\ref{eq:pbc}) for $u_0'(\alpha)$ satisfying (\ref{eq:expnew00}). Furthermore,
\begin{enumerate}
\item\label{it:lambdaneg1} Suppose $\lambda\in[-1,0)$ and $q>0$. Then, there exists a finite $t_*>0$ such that only the minimum diverges, $m(t)\to-\infty,$ as $t\uparrow t_*$ (one-sided, discrete blow-up). 
\item\label{it:lambdaneg2} Suppose $\lambda<-1$ and $q\in(0,1)$ satisfy $\lambda\neq\frac{q}{1-nq}$ and $q\neq\frac{1}{n}$\, $\forall$\, $n\in\mathbb{N}$. Then, a one-sided discrete blow-up, as described in (\ref{it:lambdaneg1}), occurs in finite-time. Similarly for $\frac{q}{1-q}<\lambda<-1$ and $q>1$. 
\item\label{it:lambdaneg3} Suppose $\lambda<\frac{q}{1-q}$ and $q>1$. Then, there is a finite $t_*>0$ such that both the maximum $M(t)$ and the minimum $m(t)$ diverge to $+\infty$ and respectively to $-\infty$ as $t\uparrow t_*$. Moreover, $\lim_{t\uparrow t_*}u_x(\gamma(\alpha,t),t)=+\infty$ for $\alpha\notin\bigcup_{i,j}\{\overline\alpha_i\}\cup\{\underline\alpha_j\}$ (two-sided, everywhere blow-up).
\end{enumerate}
Finally, for $\lambda<0$, $q>0$ and $t_*>0$ as above, the jacobian (\ref{eq:sum}) satisfies
\begin{equation}
\label{eq:lastrealjac}
\lim_{t\uparrow t_*}\gamma_{\alpha}(\alpha,t)=
\begin{cases}
0,\,\,\,\,\,\,\,\,\,\,\,\,\,\,&\alpha=\underline\alpha_j,
\\
C(\alpha),\,\,\,\,\,\,\,\,\,\,\,&\alpha\neq\underline\alpha_j
\end{cases}
\end{equation}
where $C(\alpha)\in\mathbb{R}^+$ depends on the choice of $\lambda$, $q$ and $\alpha\neq\underline\alpha_j$.

\begin{proof}
Throughout, let $C$ denote a positive constant that may depend on $\lambda<0$, $q>0$ and recall that $\eta_*=\frac{1}{\lambda m_0}.$

\vspace{0.05in}
\textbf{Proof of Statement} (\ref{it:lambdaneg1})
\vspace{0.05in}

Suppose $\lambda\in[-1,0)$ and assume $u_0'(\alpha)$ satisfies (\ref{eq:expnew00}) for some $q>0$. Then (\ref{eq:thirdint1}) and (\ref{eq:lastest}) imply that both integral terms in (\ref{eq:mainsolu}) remain finite and nonzero as $\eta\uparrow\eta_*$.\footnote[4]{Recall that $u_0'$ is assumed to be bounded and, at least, $C^0(0,1)\, a.e.$} More particularly, one can show that (\ref{eq:quick3p}) and (\ref{eq:quick4p}) hold for all $\eta\in[0,\eta_*].$ Therefore, blow-up of (\ref{eq:mainsolu}) depends, solely, on the behaviour of the space-dependent term $\mathcal{J}(\alpha,t)^{-1}$. Accordingly, we set $\alpha=\underline\alpha_j$ into (\ref{eq:mainsolu}) and use (\ref{eq:maxmin})ii) to find that the minimum diverges, $m(t)\to-\infty$, as $\eta\uparrow\eta_*$. However, if $\alpha\neq\underline\alpha_j$, the definition of $m_0$ implies that the space-dependent term now remains bounded, and positive, for $\eta\in[0,\eta_*]$. The existence of a finite blow-up time $t_*>0$ for the minimum follows from (\ref{eq:etaivp}) and (\ref{eq:thirdint1}). In fact, we may use (\ref{eq:etaivp}) and (\ref{eq:quick3p}) to obtain the estimate
\begin{equation}
\label{eq:t1}
\begin{split}
\frac{\left|m_0\right|}{\left|\lambda\right|(m_0-M_0)^2}\leq t_*\leq\eta_*.
\end{split}
\end{equation}

\vspace{0.05in}
\textbf{Proof of Statements} (\ref{it:lambdaneg2}) and (\ref{it:lambdaneg3})
\vspace{0.05in}

Now suppose $\lambda<-1$. As in  the previous case, the term $\bar{\mathcal{K}}_0(t)$ remains finite, and positive, for all $\eta\in[0,\eta_*]$. Particularly, $\bar{\mathcal{K}}_0(t)$ satisfies (\ref{eq:newestq}) for all $\eta\in[0,\eta_*]$. On the other hand, $\bar{\mathcal{K}}_1(t)$ now either converges or diverges, as $\eta\uparrow\eta_*$, according to (\ref{eq:lastest}) or (\ref{eq:lastest1}), respectively. If $\lambda<-1$ and $q>0$ are such that (\ref{eq:lastest}) holds, then part (\ref{it:lambdaneg2}) follows just as part (\ref{it:lambdaneg1}). However, if $q>1$ and $\lambda<\frac{q}{1-q}$, we use (\ref{eq:thirdint1}) and (\ref{eq:lastest1}) on (\ref{eq:mainsolu}), to obtain
$$u_x(\gamma(\alpha,t),t)\sim Cm_0\left(\frac{1}{\mathcal{J}(\alpha,t)}-C\mathcal{J}(\underline\alpha_j,t)^{\frac{1}{q}-\frac{1}{\lambda}-1}\right)$$
for $\eta_*-\eta>0$ small. Setting $\alpha=\underline\alpha_j$ into the above implies that
$$m(t)\sim\frac{Cm_0}{\mathcal{J}(\underline\alpha,t)}\to-\infty$$
as $\eta\uparrow\eta_*$, whereas, for $\alpha\neq\underline\alpha_i$, the space-dependent term now remains bounded, as a result, the second term dominates and
$$u_x(\gamma(\alpha,t),t)\sim C\left|m_0\right|\mathcal{J}(\underline\alpha_j,t)^{\frac{1}{q}-\frac{1}{\lambda}-1}\to+\infty$$
as $\eta\uparrow\eta_*$. The existence of a finite blow-up time $t_*>0$ follows as in the case $\lambda\in[-1,0)$. In fact, (\ref{eq:etaivp}) and (\ref{eq:newestq}) yield the lower bound $\eta_*\leq t_*$.\footnote[5]{Which we may compare to (\ref{eq:t1}). From (\ref{eq:etaivp}), we see that the two coincide, $t_*=\eta_*$, in the case of Burgers' equation $\lambda=-1$.} Finally, (\ref{eq:lastrealjac}) is derived straightforwardly from (\ref{eq:sum}) and (\ref{eq:thirdint1}). See \S\ref{sec:examples} for examples.\end{proof}
\end{theorem}

\subsubsection{Further $L^p$ regularity for $\lambda\in\mathbb{R}^-$, $q\in\mathbb{R}^+$ and $p\in[1,+\infty)$}\hfill
\label{subsubsec:generalcaselambdanegpfin}

Let $t_*>0$ denote the finite $L^{\infty}$ blow-up time for $u_x$ in Theorem \ref{thm:lambdaneg} above. In this last section, we briefly examine the $L^p$ regularity of $u_x$, as $t\uparrow t_*$, for $\lambda\in\mathbb{R}^-$, $p\in[1,+\infty)$ and $u_0'$ satisfying (\ref{eq:expnew00}) for some $q\in\mathbb{R}^+$. As in \S\ref{subsubsec:generalcaselambdapospfin}, we will make use of (\ref{eq:upper}) and (\ref{eq:lower}). First of all, by the last part of Lemma \ref{lem:general}(3), we have that for $q>0$ and $p\geq1$,
\begin{equation}
\label{eq:reallast}
\int_0^1{\frac{d\alpha}{\mathcal{J}(\alpha,t)^{\frac{1}{\lambda p}}}}\sim C
\end{equation}
for $\eta_*-\eta>0$ small, $\eta_*=\frac{1}{\lambda m_0}$ and $\lambda<0$. Similarly 
\begin{equation}
\label{eq:fin1}
\int_0^1{\frac{d\alpha}{\mathcal{J}(\alpha,t)^{p+\frac{1}{\lambda}}}}\sim C
\end{equation}
for $-\frac{1}{p}\leq\lambda<0$. Moreover, due to the first part of $(3)$ in the Lemma, estimate (\ref{eq:fin1}) is also seen to hold, with different positive constants $C$, for $\lambda<-\frac{1}{p}$, $p\geq1$ and $q>0$ satisfying either of the following
\begin{equation}
\label{eq:laste0}
\begin{cases}
q\in(0,1/2),\,\,\, &p\in[1,2],\,\,\,\,\,\,\,\,\,\,\,\,\,\,\,\,\,\,\,\,\,\,\,\,\, \lambda<-\frac{1}{p},
\\
q\in(0,1/2),\,\,\, &p>2,\,\,\,\,\,\,\,\,\,\,\,\,\,\,\,\, \frac{1}{2-p}<\lambda<-\frac{1}{p},
\\
q\in(1/2,1),\,\,\, &p\in[1,1/q],\,\,\,\,\,\,\,\,\,\,\,\,\,\,\,\,\,\,\, \lambda<-\frac{1}{p},
\\
q\in(1/2,1),\,\,\, &p>\frac{1}{q},\,\,\,\,\,\,\,\,\,\,\,\,\, \frac{q}{1-pq}<\lambda<-\frac{1}{p},
\\
q>1,\, &p\geq1,\,\,\,\,\,\,\,\,\,\,\,\,\, \frac{q}{1-pq}<\lambda<-\frac{1}{p},
\end{cases}
\end{equation}
as well as
\begin{equation}
\label{eq:req2lambdaneg1}
\lambda\notin\left\{\frac{q}{1-q(p+n)},\frac{1}{1-p}\right\},\,\,\,\,\,\,\,\,\,\,\,q\neq\frac{1}{n}\,\,\,\,\forall\,\,\,\,\,n\in\mathbb{N}.
\end{equation}
We remark that in the cases where (\ref{eq:fin1}) diverges, it dominates the other terms in (\ref{eq:upper}), regardless of whether these converge or diverge, and so no information on the behaviour of $\left\|u_x\right\|_p$ is obtained. Consequently, we will omit those instances. Finally,  using Lemma \ref{lem:general}(2), one finds that
\begin{equation}
\label{eq:fin2}
\int_0^1{\frac{d\alpha}{\mathcal{J}(\alpha,t)^{1+\frac{1}{\lambda p}}}}\sim C\mathcal{J}(\underline\alpha_j,t)^{\frac{1}{q}-\frac{1}{\lambda p}-1}
\end{equation}
for $q>1$, $p\geq1$ and $\lambda<\frac{q}{p(1-q)}$. Analogously, if (\ref{eq:fin2}) converges, the lower bound (\ref{eq:lower}) yields no information on the $L^p$ regularity of $u_x$. For the remaining of this section, we will assume that (\ref{eq:lemmaass1}) holds whenever (\ref{eq:lastest}) is used for $\lambda<-1$ and $q\in(0,1)$. Also, (\ref{eq:req2lambdaneg1}) will be valid in those cases where estimate (\ref{eq:fin1}) is considered for $\lambda$, $p$ and $q$ as in (\ref{eq:laste0}). Suppose $\frac{q}{1-q}<\lambda<\frac{q}{p(1-q)}$ for $q>1$ and $p>1$. Then, using (\ref{eq:thirdint1}), (\ref{eq:lastest}), (\ref{eq:reallast}) and (\ref{eq:fin2}), in the lower bound (\ref{eq:lower}), implies that
$$\lim_{t\uparrow t_*}\left\|u_x(\cdot,t)\right\|_p=+\infty.$$
If instead, $\lambda<\frac{q}{1-q}$ for $q>1$ and $p>1$, then (\ref{eq:thirdint1}), (\ref{eq:lastest1}), (\ref{eq:reallast}) and (\ref{eq:fin2}) give
\begin{equation*}
\begin{split}
\left\|u_x(\cdot,t)\right\|_p&\geq\frac{1}{\left|\lambda\eta(t)\right|\bar{\mathcal{K}}_0(t)^{^{2\lambda+\frac{1}{p}}}}\left|\int_0^1{\frac{d\alpha}{\mathcal{J}(\alpha,t)^{^{1+\frac{1}{\lambda p }}}}}-\frac{\bar{\mathcal{K}}_1(t)}{\bar{\mathcal{K}}_0(t)}\int_0^1{\frac{d\alpha}{\mathcal{J}(\alpha,t)^{\frac{1}{\lambda p}}}}\right|
\\
&\sim C\left|C\mathcal{J}(\underline\alpha,t)^{\frac{1}{q}-\frac{1}{\lambda p}-1}-\mathcal{J}(\underline\alpha,t)^{\frac{1}{q}-\frac{1}{\lambda}-1}\right|
\\
&\sim C\mathcal{J}(\underline\alpha,t)^{\frac{1}{q}-\frac{1}{\lambda p}-1}\to+\infty
\end{split}
\end{equation*}
as $\eta\uparrow\eta_*$. For the upper bound (\ref{eq:upper}), we simply mention that estimates (\ref{eq:thirdint1}), (\ref{eq:lastest}) and (\ref{eq:fin1}) lead to several instances where $\left\|u_x\right\|_p$ remains finite for all $t\in[0,t_*]$. This can be shown, just as above, by using the appropriate estimates. For simplicity, we omit the details and summarize the results in Theorem \ref{thm:genlplambdaneg} below. 
\begin{theorem}
\label{thm:genlplambdaneg}
Consider the initial boundary value problem (\ref{eq:nonhomo})-(\ref{eq:pbc}) for $u_0'(\alpha)$ satisfying (\ref{eq:expnew00}), and let $t_*>0$ denote the finite $L^{\infty}$ blow-up time for $u_x$ as described in Theorem \ref{thm:lambdaneg}. 
\begin{enumerate}
\item\label{it:plambdaneg1} Let $q\in(0,1/2)$. Then, $\lim_{t\uparrow t_*}\left\|u_x\right\|_p<+\infty$ for either $\lambda<0$ and $p\in[1,2]$, or $\frac{1}{2-p}<\lambda<0$ and $p>2$.
\item\label{it:plambdaneg2} Let $q\in(1/2,1)$. Then, $\lim_{t\uparrow t_*}\left\|u_x\right\|_p<+\infty$ for either $\lambda<0$ and $p\in[1,1/q]$, or $\frac{q}{1-pq}<\lambda<0$ and $p>1/q$.
\item\label{it:plambdaneg3} Let $q>1$. Then $\lim_{t\uparrow t_*}\left\|u_x\right\|_p<+\infty$ for $\frac{q}{1-pq}<\lambda<0$ and $p\geq1$, whereas $\lim_{t\uparrow t_*}\left\|u_x\right\|_p=+\infty$ for $\lambda<\frac{q}{p(1-q)}$ and $p>1$.
\end{enumerate}
When applicable, (\ref{eq:lemmaass1}) and (\ref{eq:req2lambdaneg1}) apply to (\ref{it:plambdaneg1}) and (\ref{it:plambdaneg2}) above.
\end{theorem}

\subsubsection{Further regularity results for smooth initial data}
\label{subsubsec:arbitrarysmooth}

\begin{definition}
\label{def:order}
Suppose a smooth function $f(x)$ satisfies $f(x_0)=0$ but $f$ is not identically zero. We say $f$ has a zero of order $k\in\mathbb{N}$ at $x=x_0$ if 
$$f(x_0)=f'(x_0)=...=f^{(k-1)}(x_0)=0,\,\,\,\,\,\,\,\,\,\,\,\,f^{(k)}(x_0)\neq0.$$
\end{definition}
In \cite{Sarria1}, we examined a class of smooth, mean-zero, periodic initial data characterized by $u_0''(\alpha)$ having zeroes of order $k=1$ at the finite number of locations $\overline\alpha_i$ for $\lambda>0$, or at $\underline\alpha_j$ if $\lambda<0$, that is, $u_0'''(\overline\alpha_i)<0$ or $u_0'''(\underline\alpha_j)>0$. Consequently, in each case, we were able to use an appropriate Taylor expansion up to quadratic order to account for the local behaviour of $u_0'$ near these points. This approach, in turn, led to the results summarized in Theorems \ref{thm:sarria1} and \ref{thm:lpintro} of \S\ref{sec:intro}. Assuming the order $k$ of these particular zeroes, $\overline\alpha_i$ or $\overline\alpha_j$, of $u_0''$ is the same regardless of location, and noticing that $k\geq1$ must be odd due to $u_0'$ being even in a small neighbourhood of these points, we may use definition \ref{def:order} to generalize the results in \cite{Sarria1} to a larger class of smooth, mean-zero, periodic initial data characterized by $u_0''$ having zeroes of higher orders, $k=1,3,5,...$, at every $\overline\alpha_i$ if $\lambda>0$, or $\underline\alpha_j$ for $\lambda<0$. Since this corresponds to replacing $q$ in (\ref{eq:expnew0}) or (\ref{eq:expnew00}) by $k+1$, we obtain our results simply by substituting $q$ in Theorems \ref{thm:lambdapos}, \ref{thm:lplambdapos}, \ref{thm:lambdaneg} and \ref{thm:genlplambdaneg}, by $1+k$ in those cases where $q\geq2$. The results are summarized in Corollary \ref{coro:gensmooth} below. 

\begin{corollary}
\label{coro:gensmooth}
Consider the initial boundary value problem (\ref{eq:nonhomo})-(\ref{eq:pbc}) for smooth, mean-zero, periodic initial data. Furthermore,

\begin{enumerate}
\item Suppose $u_0''(\alpha)$ has a zero of order $k\geq1$ at every $\overline\alpha_i$, $i=1,2,...,m$. Then

\begin{itemize}

\item\label{it:globalgen} For $0\leq\lambda\leq\frac{1+k}{2}$, solutions exist globally in time. More particularly, these vanish as $t\uparrow t_*=+\infty$ for $0<\lambda<\frac{1+k}{2}$ but converge to a nontrivial steady state if $\lambda=\frac{1+k}{2}$.

\item\label{it:blow1gen} For $\frac{1+k}{2}<\lambda<+\infty$, there exists a finite $t_*>0$ such that both the maximum $M(t)$ and the minimum $m(t)$ diverge to $+\infty$ and respectively to $-\infty$ as $t\uparrow t_*$. Furthermore, $\lim_{t\uparrow t_*}u_x(\gamma(\alpha,t),t)=-\infty$ if $\alpha\notin\bigcup_{i,j}\{\overline\alpha_i\}\cup\{\underline\alpha_j\}$ and $\lim_{t\uparrow t_*}\left\|u_x\right\|_{p}=+\infty$ for all $p>1$.

\end{itemize}

\item Suppose $u_0''(\alpha)$ has a zero of order $k\geq1$ at each $\underline\alpha_j$, $j=1,2,...,n$. Then

\begin{itemize}

\item\label{it:lambdaneg1gen} For $-\frac{1+k}{k}<\lambda<0$, there exists a finite $t_*>0$ such that only the minimum diverges, $m(t)\to-\infty,$ as $t\uparrow t_*$, whereas, for $\frac{1+k}{1-p(1+k)}<\lambda<0$ and $p\geq1$, $\lim_{t\uparrow t_*}\left\|u_x\right\|_p<+\infty$.

\item\label{it:lambdaneg3gen} For $\lambda<-\frac{1+k}{k}$, there is a finite $t_*>0$ such that both $M(t)$ and $m(t)$ diverge to $+\infty$ and respectively to $-\infty$ as $t\uparrow t_*$. Additionally, $\lim_{t\uparrow t_*}u_x(\gamma(\alpha,t),t)=+\infty$ for $\alpha\notin\bigcup_{i,j}\{\overline\alpha_i\}\cup\{\underline\alpha_j\}$ and $\lim_{t\uparrow t_*}\left\|u_x\right\|_p=+\infty$ if $\lambda<-\frac{1+k}{pk}$ and $p>1$.

\end{itemize}

\end{enumerate}
\end{corollary}

\begin{remark}
\label{rem:exception}
It turns out that, unless the initial data is smooth, the results established in this paper for periodic boundary conditions extend to a Dirichlet setting. For smooth initial data, if there are $\overline\alpha_i\in\{0,1\}$ when $\lambda>0$, or $\underline\alpha_j\in\{0,1\}$ for $\lambda<0$, then the results in the periodic setting will extend to Dirichlet boundary conditions as long as $u_0''$ vanishes at those end-points. This last condition prevents a Lipschitz-type behaviour of $u_0'$ at the boundary, which could otherwise lead to finite-time blow-up from smooth initial data under Dirichlet boundary conditions. Details on this will be presented in an upcoming paper. Also, notice that letting $q\to+\infty$ in either (\ref{eq:expnew0}) or (\ref{eq:expnew00}) implies that $u_0'\sim M_0$ near $\overline\alpha_i$, or $u_0'\sim m_0$ for $\alpha\sim\underline\alpha_j$, respectively. Then, letting $k\to+\infty$ in Corollary \ref{coro:gensmooth}(\ref{it:globalgen}) implies that, for this particular class of locally constant $u_0'$, a solution that exists locally in time for any $\lambda\in\mathbb{R}$, will persist for all time. 
\end{remark}

\section{Examples}
\label{sec:examples}

Examples for Theorems \ref{thm:p=1}, \ref{thm:lambdapos} and \ref{thm:lambdaneg} are now presented. For simplicity, we consider initial data satisfying Dirichlet boundary conditions\footnote[6]{The reader may refer to \cite{Sarria1} for examples involving periodic, mean-zero data satisfying (\ref{eq:expnew0}), and/or (\ref{eq:expnew00}), for $q=2$ or in the limit as $q\to+\infty$.}
$$u(0,x)=u(1,t)=0,$$
and we note that (\ref{eq:mainsolu}) is equivalent to the representation formula (see \cite{Sarria1})
\begin{equation}
\label{eq:finalsolu}
\begin{split}
u_x(\gamma(\alpha,t),t)=\frac{1}{\bar{\mathcal{K}}_0(t)^{^{2\lambda}}}\left(\frac{u_0'(\alpha)}{\mathcal{J}(\alpha,t)}-\frac{1}{\bar{\mathcal{K}}_0(t)}\int_0^1{\frac{u_0'(\alpha)d\alpha}{\mathcal{J}(\alpha,t)^{1+\frac{1}{\lambda}}}}\right).
\end{split}
\end{equation}
For several choices of $\lambda\in\mathbb{R}$, the time-dependent integrals in (\ref{eq:finalsolu}) are evaluated and pointwise plots are generated using \textsc{Mathematica}. Whenever possible, plots in the Eulerian variable $x,$ instead of the Lagrangian coordinate $\alpha,$ are provided. For practical reasons, details of the computations in most examples are omitted. Also, due to the difficulty in solving for the time variable $t$ through the IVP (\ref{eq:etaivp}) for $\eta(t)$, most plots for $u_x(\gamma(\alpha,t),t)$ are against the variable $\eta$ rather than $t$.

Example 1 below applies to stagnation point-form (SPF) solutions to the incompressible 3D Euler equations ($\lambda=1/2$). We consider two types of data, one satisfying (\ref{eq:expnew0}) for $q\in(0,1)$, and other having $q>1$. Recall from Table \ref{table:thecase} that if $q\geq1$, global existence in time follows, while, for $q\in(1/2,1)$, finite-time blow-up occurs. Below, we see that a spontaneous singularity may also form if $q=1/3$.

\textbf{Example 1.\, Regularity of SPF solutions to 3D Euler for $q=1/3$ and $q=6/5$} 

First, for $\lambda=1/2$ and $\alpha\in[0,1]$, let
\begin{equation}
\label{eq:data3d}
\begin{split}
u_0(\alpha)=\alpha(1-\alpha^{\frac{1}{3}}).
\end{split}
\end{equation}
Then $u_0'(\alpha)=1-\frac{4}{3}\alpha^{\frac{1}{3}}$ achieves its maximum $M_0=1$ at $\overline\alpha=0$. Also, $q=1/3$, $\eta_*=2$, and $u_0'(\alpha)\notin C^1(0,1)$, i.e. $\lim_{\alpha\downarrow0}u_0''(\alpha)=-\infty$; a jump discontinuity of infinite magnitude in $u_0''$. Evaluating the integrals in (\ref{eq:finalsolu}), we obtain
\begin{equation}
\label{eq:3d1}
\begin{split}
\bar{\mathcal{K}}_0(t)=-\frac{54(\eta(t)-6)\eta(t)-81(2-\eta(t))(6+\eta(t))\,\text{arctanh}\left(\frac{2\eta(t)}{\eta(t)-6}\right)}{4(6+\eta(t))\eta(t)^3}
\end{split}
\end{equation}
and
\begin{equation}
\label{eq:3d2}
\begin{split}
&\int_0^1{\frac{u_0'(\alpha)\,d\alpha}{\mathcal{J}(\alpha,t)^{3}}}=-\frac{27\left(9(2-\eta(t))(6+\eta(t))^2\,\log\left(\frac{24}{\eta(t)+6}-3\right)\right)}{8(6+\eta(t))^2\eta(t)^4}
\\
&\,\,\,\,\,\,\,\,\,\,\,\,\,\,\,\,\,\,\,\,\,-\frac{27\left(8\eta(t)(54-(\eta(t)-9)\eta(t))+6\eta(t)(6+\eta(t))^2\,\text{arctanh}\left(\frac{2\eta(t)}{\eta(t)-6}\right)\right)}{8(6+\eta(t))^2\eta(t)^4}
\end{split}
\end{equation}
for $0\leq\eta<2$. Furthermore, in the limit as $\eta\uparrow\eta_*=2$, $\bar{\mathcal{K}}_0(t_*)=27/16$
whereas $\int_0^1{\frac{u_0'(\alpha)\,d\alpha}{\mathcal{J}(\alpha,t)^{3}}}\to+\infty.$ Also, (\ref{eq:etaivp}) and (\ref{eq:3d1}) yield
$$t(\eta)=-\frac{9\left(2\eta(6-5\eta)+9(\eta-2)^2\text{arctanh}\left(\frac{2\eta}{\eta-6}\right)\right)}{16\eta^2},$$
so that $t_*=\lim_{\eta\uparrow2}t(\eta)=9/4.$ Using (\ref{eq:3d1}) and (\ref{eq:3d2}) on (\ref{eq:finalsolu}), we find that $u_x(\gamma(\alpha,t),t)$ undergoes a two-sided, everywhere blow-up as $t\uparrow 9/4$. 

Next, replace $q=1/3$ in (\ref{eq:data3d}) by $q=6/5$. Then, $u_0'(\alpha)=1-\frac{11}{5}\alpha^{\frac{6}{5}}$ so that $u_0''$ is now defined as $\alpha\downarrow0$. Also, for this data, both integrals now diverge to $+\infty$ as $\eta\uparrow2$. Particularly, this causes a balancing effect amongst the terms in (\ref{eq:finalsolu}) that was previously absent when $q=1/3$. Ultimately, we find that as $t\to t_*=+\infty$, $u_x(\gamma(\alpha,t),t)\to0$ for every $\alpha\in[0,1]$. See Figure \ref{fig:ex7} below.

\begin{center}
\begin{figure}[!ht]
\includegraphics[scale=0.33]{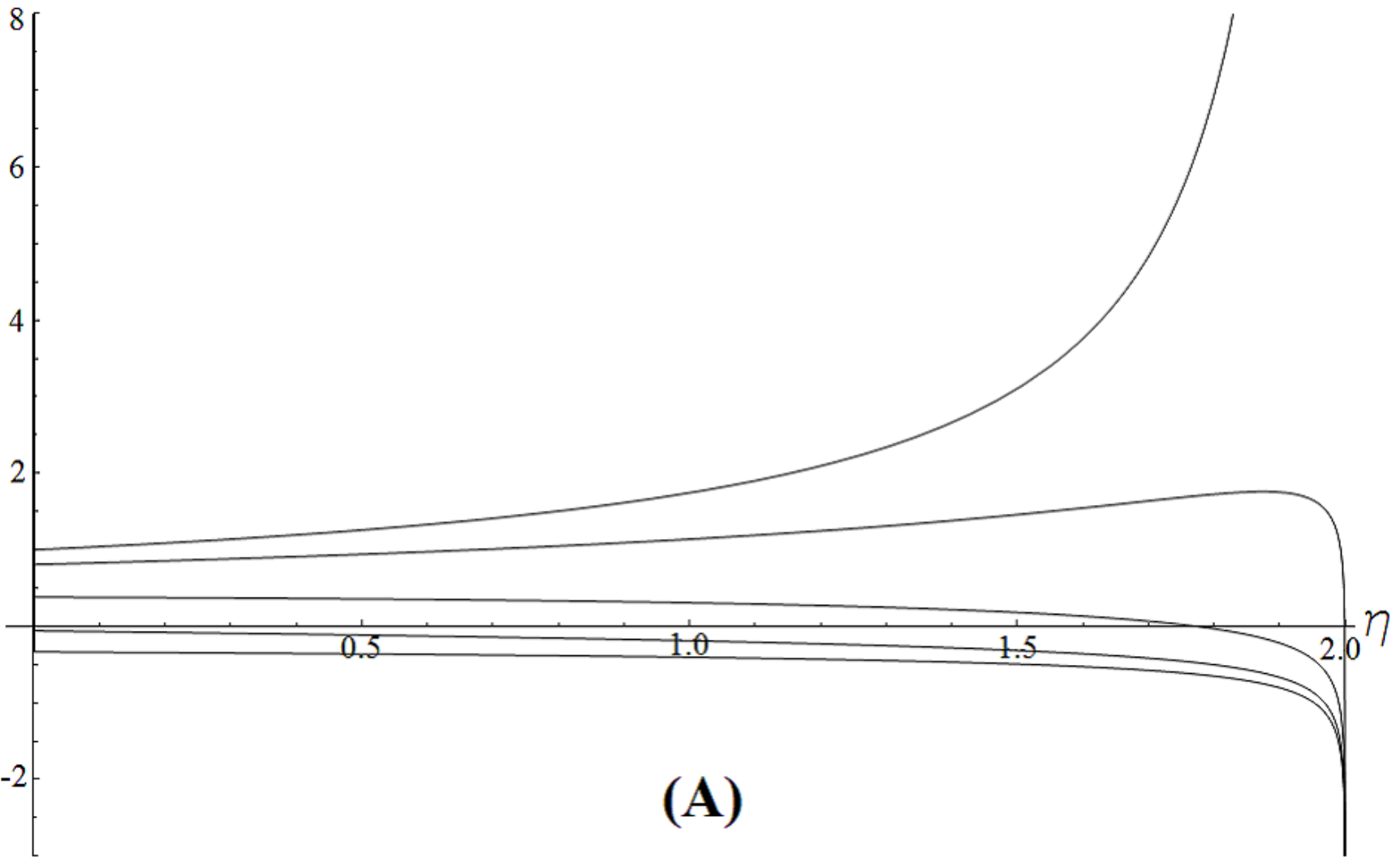}
\includegraphics[scale=0.32]{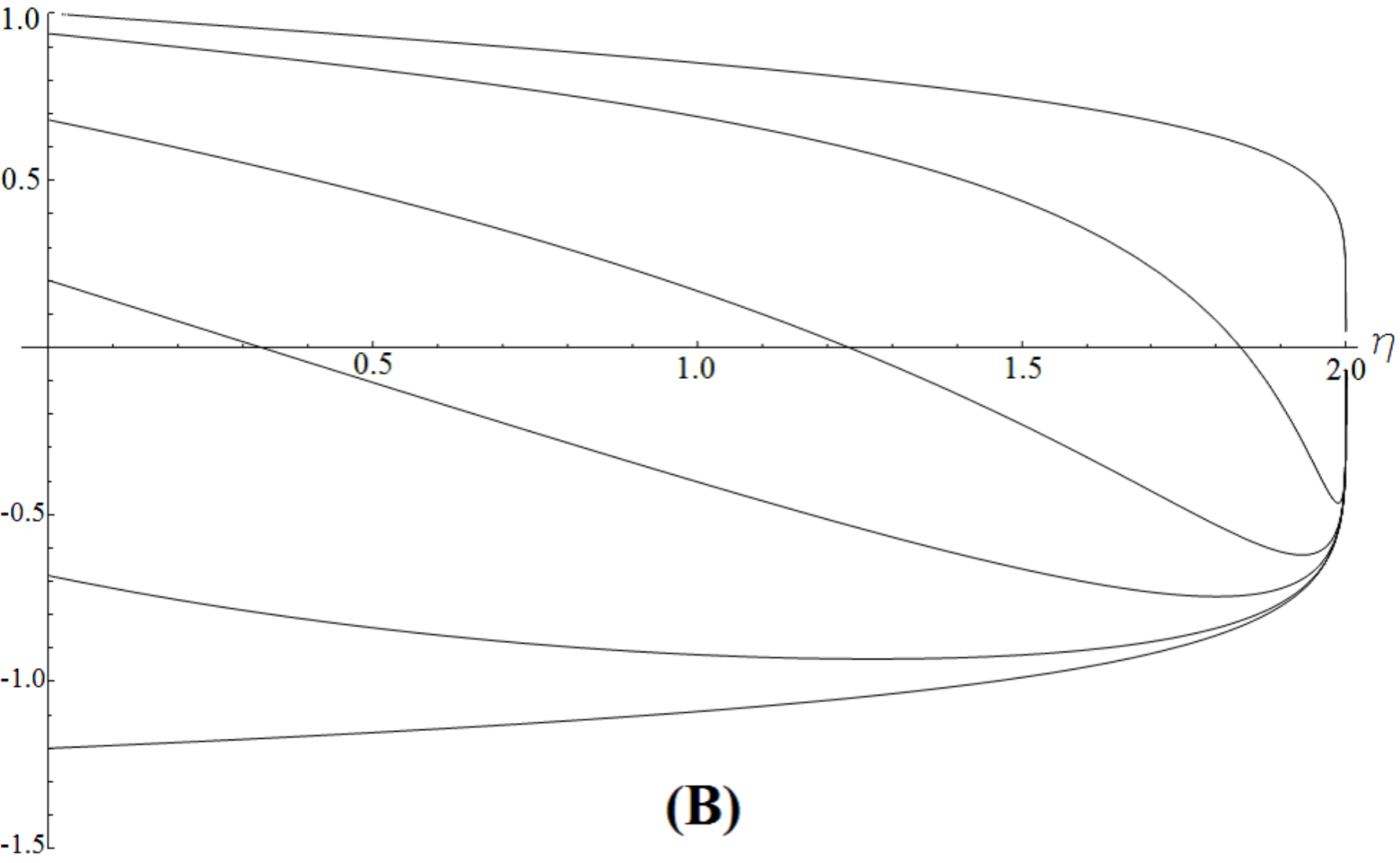}\\
\caption{Example 1 for $\lambda=1/2$ and $q\in\{1/3,6/5\}$. Figure $A$ depicts two-sided, everywhere blow-up of $u_x(\gamma(\alpha,t),t)$ for $q=1/3$ as $\eta\uparrow 2$ ($t\uparrow9/4$), whereas, for $q=6/5$, Figure $B$ represents its vanishing as $\eta\uparrow 2$ ($t\to+\infty$).}
\label{fig:ex7}
\end{figure}
\end{center}

In \cite{Sarria1} (see Theorem \ref{thm:sarria1} in \S\ref{sec:intro}), we showed that for a class of smooth, periodic initial data $(q=2)$, finite-time blow-up occurs for all $\lambda>1$. Example 2 below is an instance of Theorem \ref{thm:lambdapos}(\ref{it:global}). For $\lambda\in\{2,5/4\}$, we consider initial data satisfying (\ref{eq:expnew0}) for $q\in\{5,5/2\}$, respectively, and find that solutions persist globally in time. Also, the example illustrates the two possible global behaviours: convergence of solutions, as $t\to+\infty$, to nontrivial or trivial steady states. 

\textbf{Example 2.\, Global existence for $\lambda=2$, $q=5$ and $\lambda=5/4$, $q=5/2$}

First, let $\lambda=2$ and 
\begin{equation}
\label{eq:dat00}
\begin{split}
u_0(\alpha)=\alpha(1-\alpha^5).
\end{split}
\end{equation}
Then $u_0^\prime(\alpha)=1-6\alpha^5$ achieves its greatest value $M_0=1$ at $\overline\alpha=0$ and $\eta_*=1/2$. Since $\lambda=2\in[0,5/2)=[0,q/2)$, Theorem \ref{thm:lambdapos}(\ref{it:global}) implies global existence in time. Particularly, $u_x(\gamma(\alpha,t),t)\to0$ as $t\to+\infty.$ See Figure \ref{fig:ex34}$(A)$. Now, suppose $\lambda=5/4$ and replace $q=5$ in (\ref{eq:dat00}) by $q=5/2$. Then, $u_0'(\alpha)=1-\frac{7}{2}\alpha^{5/2}$ attains $M_0=1$ at $\overline\alpha=0$ and $\eta_*=4/5$. Because $\lambda=5/4=q/2$, Theorem \ref{thm:lambdapos}(\ref{it:global}) implies that $u_x$ converges to a nontrivial steady-state as $t\to+\infty$. See Figure \ref{fig:ex34}$(B)$.

\begin{center}
\begin{figure}[!ht]
\includegraphics[scale=0.37]{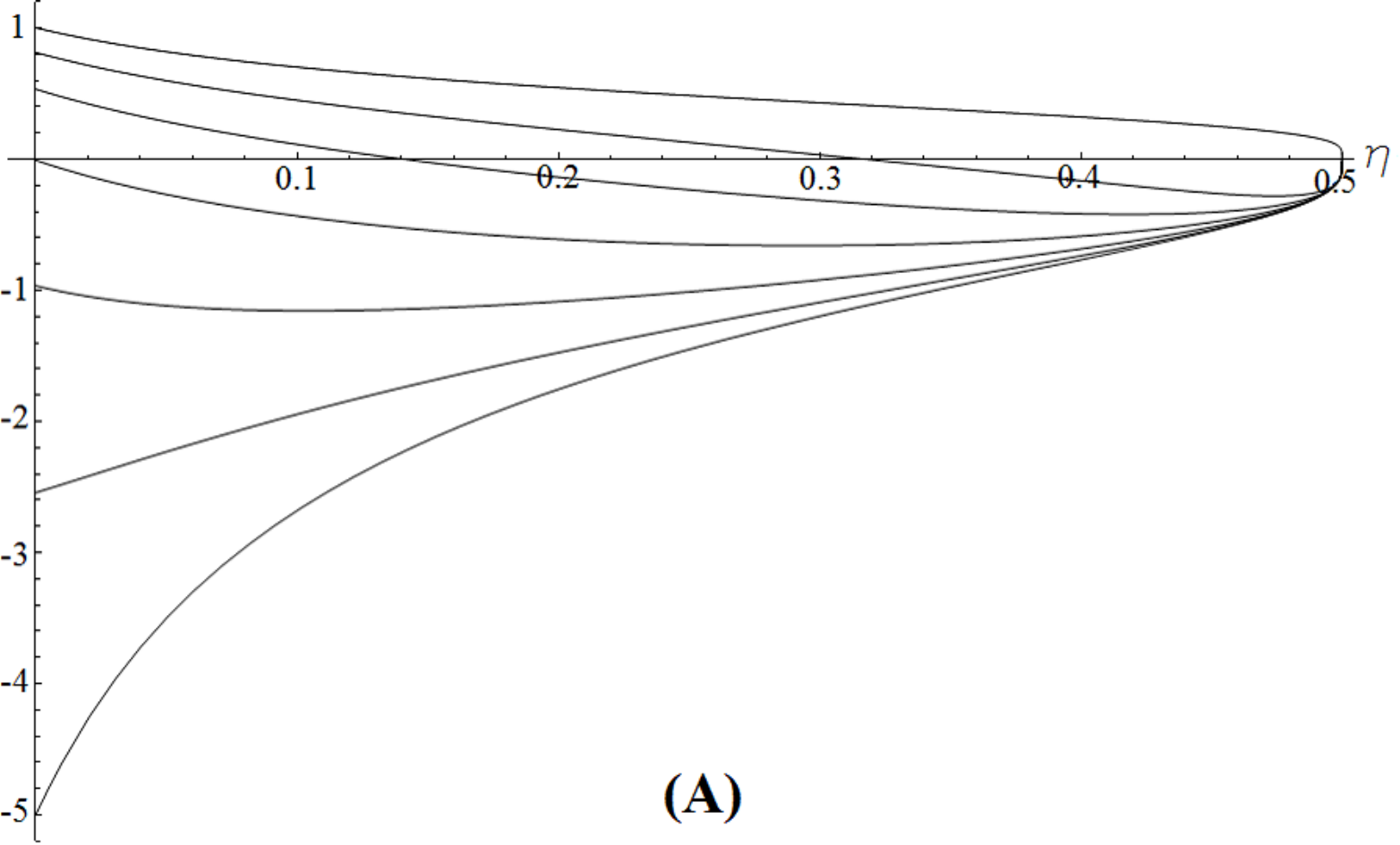}
\includegraphics[scale=0.34]{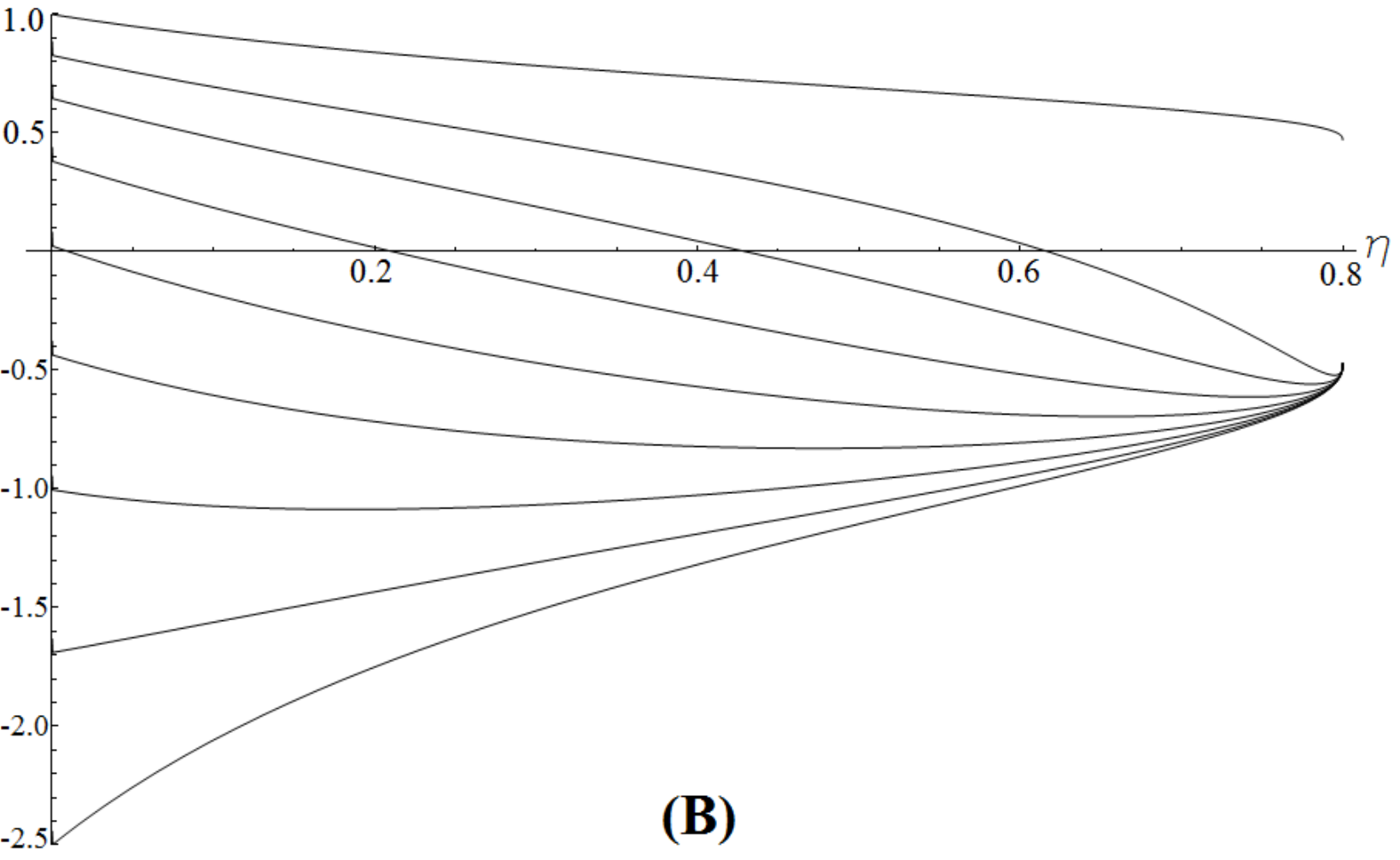} 
\caption{For example 2, Figure $A$ represents the vanishing of $u_x(\gamma(\alpha,t),t)$ as $\eta\uparrow1/2$ ($t\to+\infty$) for $\lambda=2$ and $q=5$, whereas, Figure $B$ illustrates its convergence to a nontrivial steady state as $\eta\uparrow4/5$ ($t\to+\infty$) if $q=5/2$ and $\lambda=5/4=q/2$.}
\label{fig:ex34}
\end{figure}
\end{center}

\textbf{Example 3.\, Two-sided, everywhere blow-up for $\lambda=\frac{11}{2}$ and $q=6$.}

Suppose $\lambda=11/2$ and $u_0(\alpha)=\frac{\alpha}{11}(1-\alpha^6)$. Then, $u_0'(\alpha)=\frac{1}{11}(1-7\alpha^6)$ attains its greatest value $M_0=1/11$ at $\overline\alpha=0.$ Also, $\eta_*=2$ and $\lambda=11/2\in(q/2,q)$. According to Theorem \ref{thm:lambdapos}(\ref{it:blow1}), two-sided, everywhere finite-time blow-up occurs. The estimated blow-up time is $t_*\sim 22.5.$ See Figure \ref{fig:ex340}$(A)$.

\textbf{Example 4.\, One-sided, discrete blow-up for $\lambda=-5/2$ and $q=3/2$}

Let $\lambda=-5/2$ and $u_0(\alpha)=\alpha(\alpha^{\frac{3}{2}}-1)$. Then $u_0'$ attains its minimum $m_0=-1$ at $\underline\alpha=0$ and $\eta_*=2/5.$ Since $\frac{q}{1-q}<\lambda<-1,$ Theorem \ref{thm:lambdaneg}(2) implies one-sided, discrete finite-time blow-up and $t_*\sim0.46$. See Figure \ref{fig:ex340}$(B)$. We remark that in \cite{Sarria1}, the same value for $\lambda$ with smooth, periodic initial data, and $q=2$ led to two-sided, everywhere blow-up instead.

\begin{center}
\begin{figure}[!ht]
\includegraphics[scale=0.34]{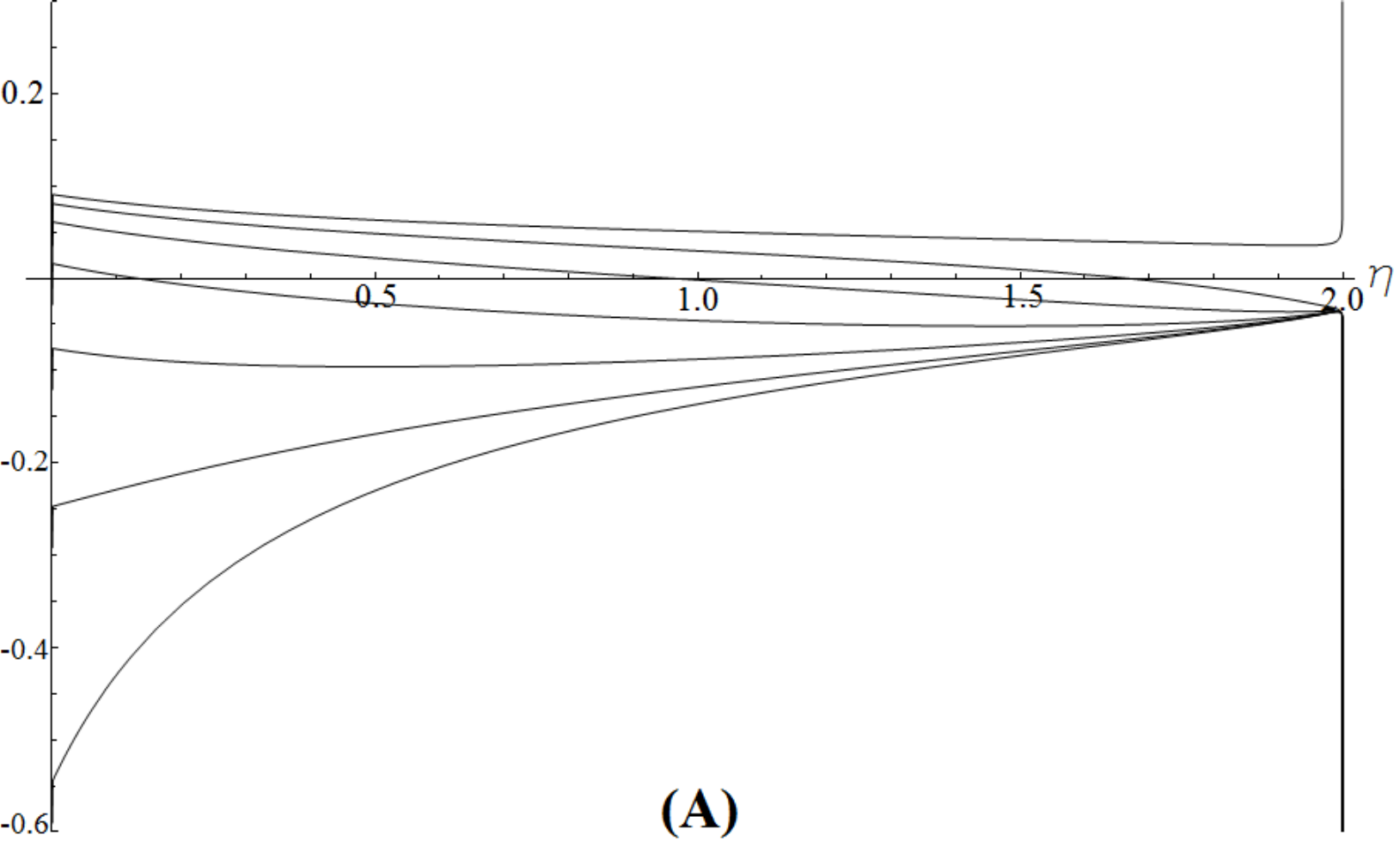} 
\includegraphics[scale=0.33]{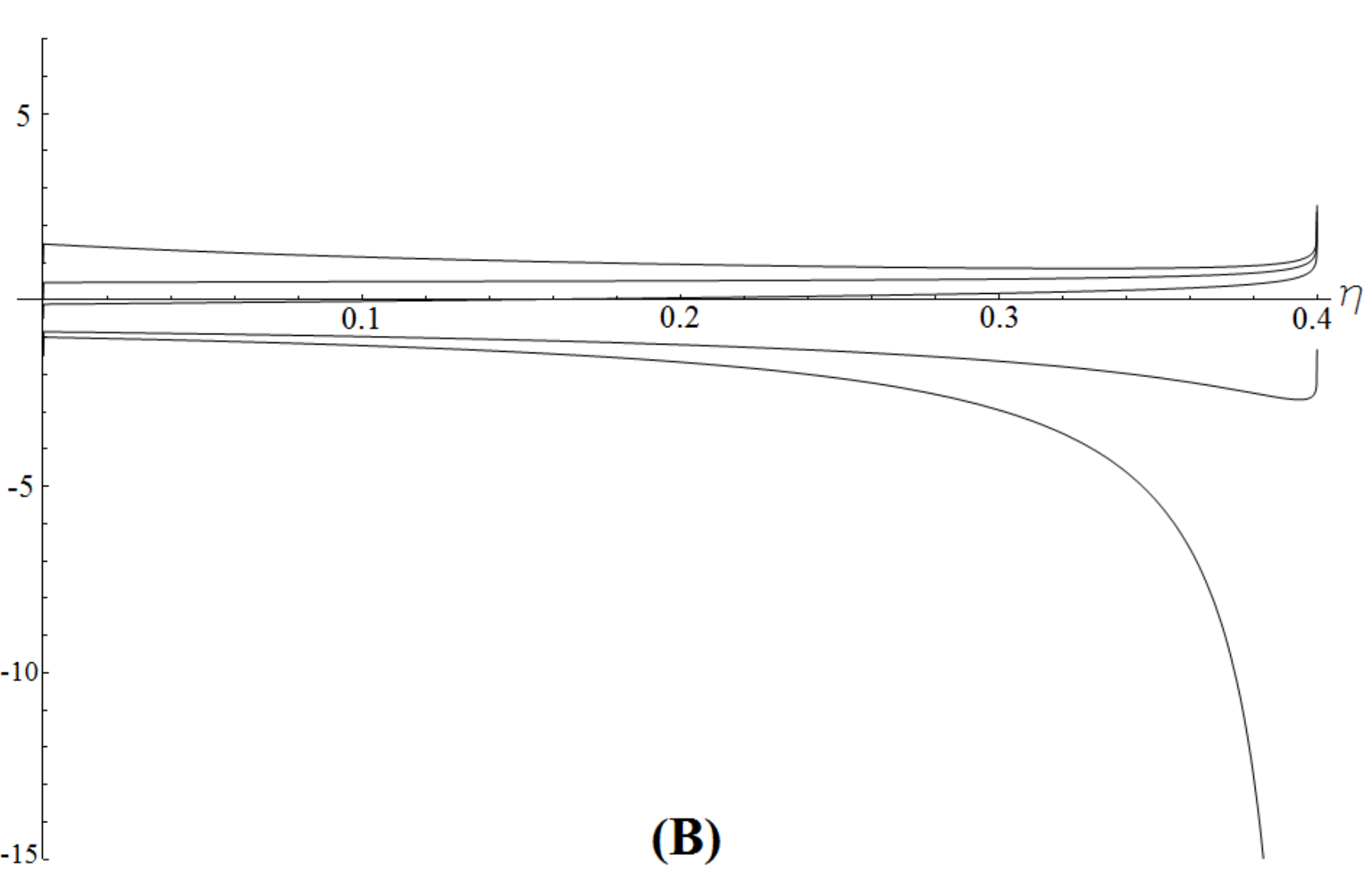}
\caption{Figure $A$ for example 3 depicts two-sided, everywhere blow-up of $u_x(\gamma(\alpha,t),t)$ as $\eta\uparrow2$ $(t\uparrow22.5)$ for $\lambda=11/2$ and $q=6$, while, Figure $B$ for example 4 illustrates one-sided, discrete blow-up, $m(t)=u_x(0,t)\to-\infty$, as $\eta\uparrow2/5$ ($t\uparrow t_*\sim0.46$) for $\lambda=-5/2$ and $q=3/2$.}
\label{fig:ex340}
\end{figure}
\end{center}

In these last two examples, we consider smooth data with either mixed local behaviour near two distinct locations $\underline\alpha_j$ for $\lambda=-1/3$, or $M_0$ occurring at both endpoints for $\lambda=1$.

\textbf{Example 5.\, One-sided, discrete blow-up for $\lambda=-1/3$ and $q=1,2$.} 
								
For $\lambda=-1/3$, let $$u_0(\alpha)=\alpha(1-\alpha)(\alpha-\frac{3}{4})\left(\alpha-\frac{1+4\sqrt{22}}{36}\right).$$
Then $m_0\sim-0.113$ occurs at both $\underline\alpha_1=1$ and $\underline\alpha_2=\frac{4+\sqrt{22}}{24}\sim0.36.$ Now, near $\underline\alpha_2$, $u_0'$ behaves quadratically ($q=2$), whereas, for $1-\alpha>0$ small, it behaves linearly ($q=1$). The quadratic behaviour is due to $u_0''$ having zero of order one at $\underline\alpha_2\sim0.36$, thus, Corollary \ref{coro:gensmooth} implies a discrete, one-sided blow-up. Similarly in the case of linear behaviour according to Theorem \ref{thm:p=1}. After evaluating the integrals, we find that $m(t)\to-\infty$ as $t\uparrow t_*\sim17.93$. Due to the Dirichlet boundary conditions, we have that $\gamma(0,t)\equiv0$ and $\gamma(1,t)\equiv1$ for a s long a s $u$ is defined. Then, one blow-up location is given by the boundary $\underline{x}_1=1$, while the interior blow-up location, $\underline x_2$, is obtained by integrating (\ref{eq:sum}). This yields the characteristics:
$$\gamma(\alpha,t)=\int_0^{\alpha}{\frac{dy}{\mathcal{J}(y,t)^{\frac{1}{\lambda}}}}\left(\int_0^{1}{\frac{d\alpha}{\mathcal{J}(\alpha,t)^{\frac{1}{\lambda}}}}\right)^{-1}.$$
Setting $\alpha=\alpha_2$ and letting $\eta\uparrow\eta_*=\frac{3}{\left|m_0\right|}$, we find that $\underline{x}_2\sim0.885$. See Figure \ref{fig:weird}$(A)$.

\textbf{Example 6.\, Two-sided, everywhere blow-up of SPF solutions to 2D Euler ($\lambda=1$) for $q=1$.} 

For $\lambda=1$, let $u_0(\alpha)=\alpha(\alpha-1)(\alpha-1/2).$ Then, $M_0=1/2$ occurs at both endpoints $\overline\alpha_i=\{0,1\}$. Also $\eta_*=2$ and since
$$u_0'(\alpha)=M_0-3\alpha+3\alpha^2=M_0-3\left|\alpha-1\right|+3(\alpha-1)^2,$$
the local behaviour of $u_0'$ near both endpoints is linear ($q=1$). The integrals in (\ref{eq:finalsolu}) evaluate to
$$\bar{\mathcal{K}}_0(t)=\frac{2\,\text{arctanh}(y(t))}{\sqrt{3\eta(t)(4+\eta(t))}},\,\,\,\,\,\,\,\,\,\,\int_0^1{\frac{u_0'(\alpha)\,d\alpha}{\mathcal{J}(\alpha,t)^2}}=\frac{\,\,d\bar{\mathcal{K}}_0(t)}{d\eta}$$
for $0\leq\eta<2$ and $y(t)=\frac{\sqrt{3\eta(t)(4+\eta(t))}}{2(1+\eta(t))}$. Using the above on (\ref{eq:finalsolu}), we find that $M(t)=u_x(0,t)=u_x(1,t)\to+\infty$ as $\eta\uparrow2$, while $u_x(x,t)\to-\infty$ for all $x\in(0,1)$. The blow-up time is estimated from (\ref{eq:etaivp}) and $\bar{\mathcal{K}}_0(t)$ above as $t_*\sim2.8$. See Figure \ref{fig:weird}$(B)$.

\begin{center}
\begin{figure}[!ht]
\includegraphics[scale=0.34]{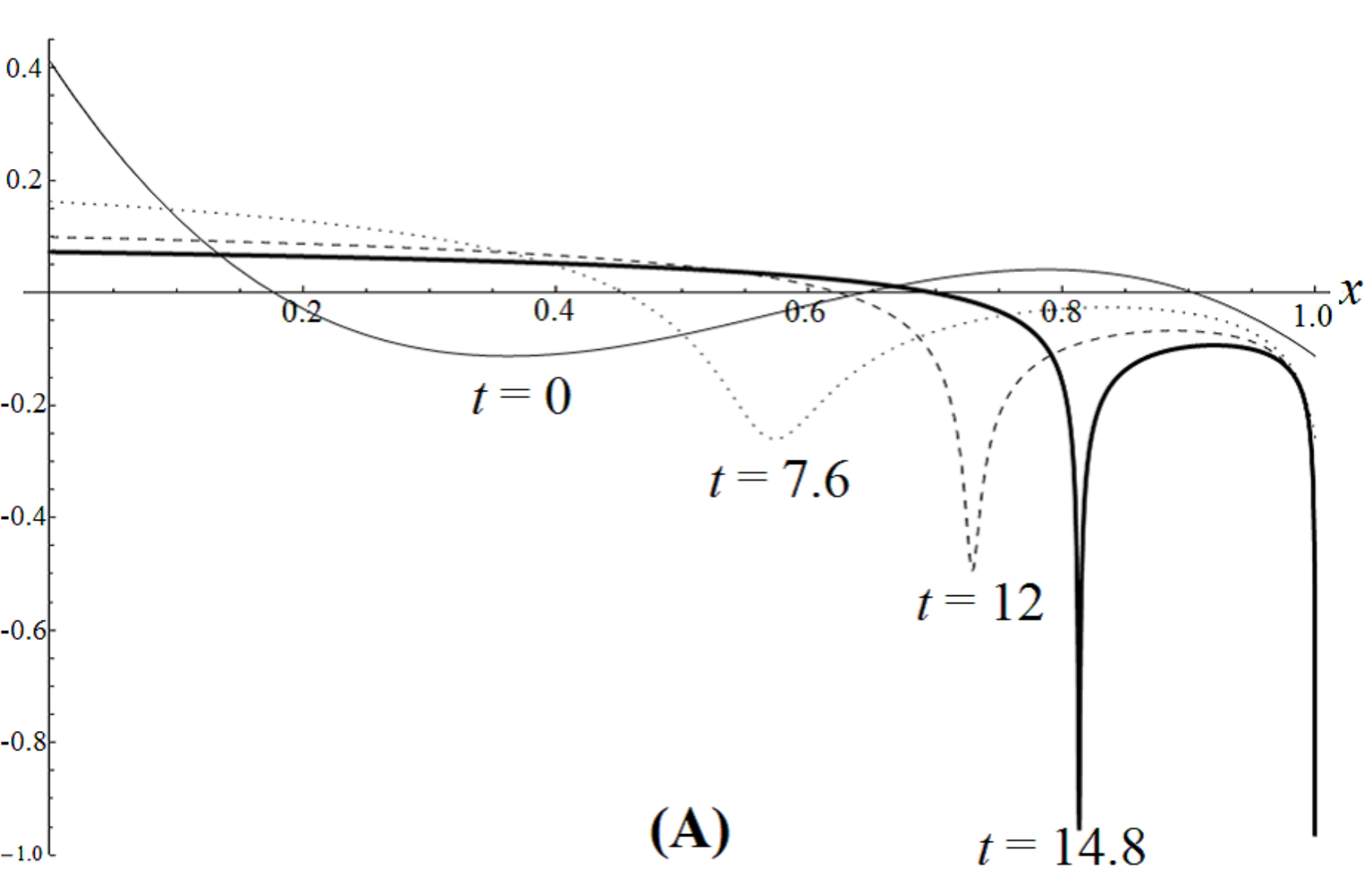}
\includegraphics[scale=0.33]{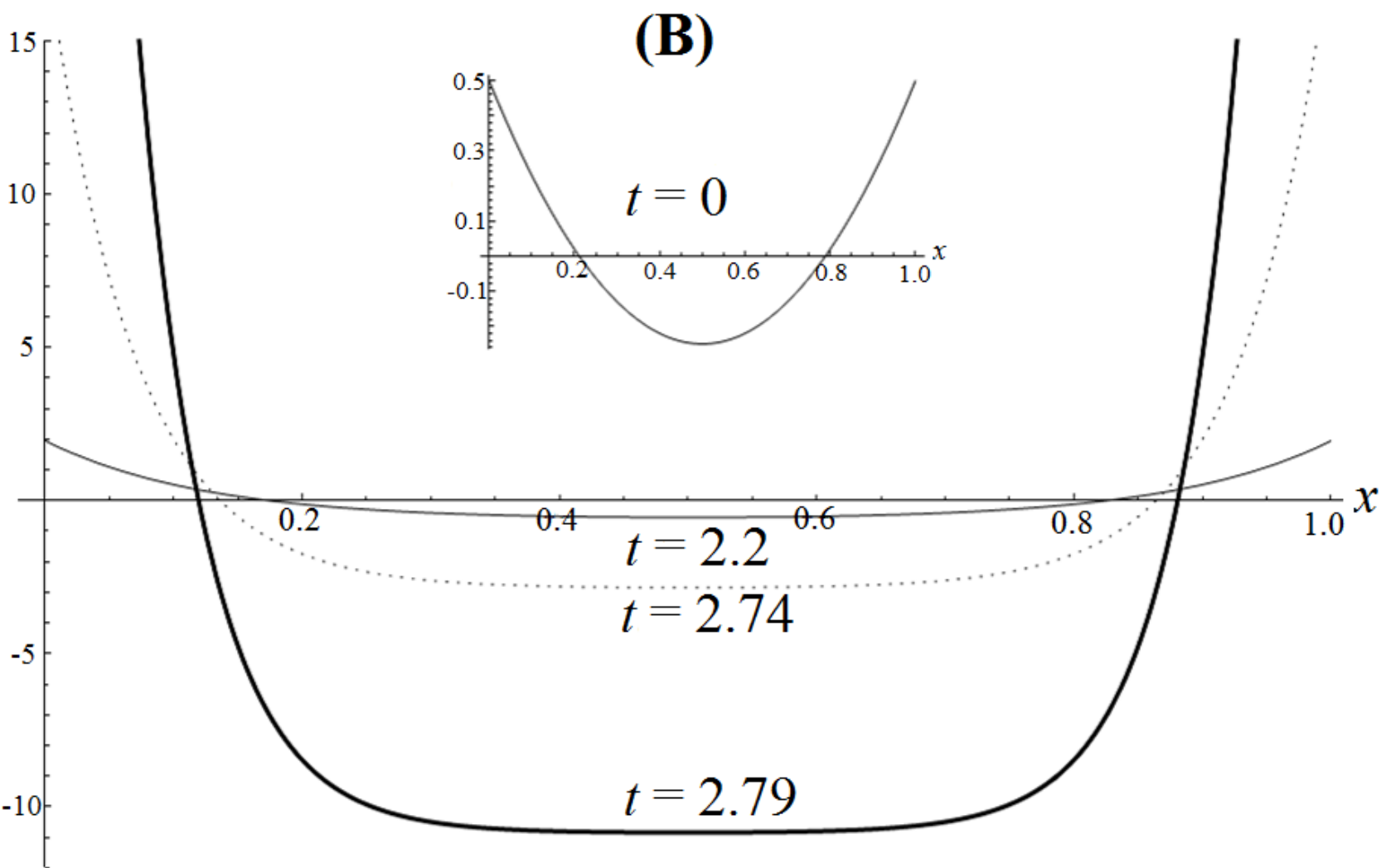}
\caption{Figure $A$ for example 5 with $\lambda=-1/3$ and $q=1,2$, depicts one-sided, discrete blow-up, $m(t)=\to-\infty$, as $t\uparrow17.93$. The blow-up locations are $\underline{x}_1=1$ and $\underline{x}_2\sim0.885$. Then, Figure $B$ for example 6 with $\lambda=1$ and $q=1$, represents two-sided, everywhere blow-up of $u_x(x,t)$, as $t\uparrow2.8$.}
\label{fig:weird}
\end{figure}
\end{center}


\end{document}